\documentclass{amsart}
\usepackage{amsmath}
\usepackage{amsthm}
\usepackage{amssymb}
\usepackage{mathrsfs}
\usepackage{hyperref}
\usepackage{pifont}
\usepackage{tikz}
\usetikzlibrary{matrix,arrows}
\usepackage{stackengine}
\stackMath
\setstackgap{S}{1pt}
\usepackage{wasysym}

\makeatletter
\newcommand{\eqnum}{\refstepcounter{equation}\textup{\tagform@{\theequation}}}
\makeatother

\numberwithin{equation}{section}

\makeatletter
\newtheorem*{rep@theorem}{\rep@title}
\newcommand{\newreptheorem}[2]{%
\newenvironment{rep#1}[1]{%
 \def\rep@title{#2 \ref{##1}}%
 \begin{rep@theorem}}%
 {\end{rep@theorem}}}
\makeatother

\newtheorem{thm}{Theorem}[section]
\newreptheorem{thm}{Theorem}
\newtheorem{prop}[thm]{Proposition} 
\newreptheorem{prop}{Proposition}

\newtheorem{lem}[thm]{Lemma}
\newreptheorem{lem}{Lemma}
\newtheorem{cor}[thm]{Corollary}
\newreptheorem{cor}{Corollary}

\theoremstyle{definition}
\newtheorem{dfn}[thm]{Definition}
\newtheorem{exmpl}[thm]{Example}
\newtheorem{?}[thm]{Question}

\newtheorem*{b?}{Question}
\newtheorem*{prob*}{Problem}

\theoremstyle{remark}
\newtheorem{rmk}[thm]{Remark}

\newcommand{\HOM}{\mathbb{H}\text{om}}
\newcommand{\FR}{\mathfrak}
\newcommand{\ds}{\displaystyle}
\newcommand{\tql}{\textquotedblleft}
\newcommand{\tqr}{\textquotedblright}
\newcommand{\noin}{\noindent}
\newcommand{\mc}{\mathcal}
\newcommand{\mb}{\mathbb}

\newcommand{\ms}{\mathscr}

\newcommand{\floor}[1]{\lfloor #1 \rfloor}

\begin{document}
\title[Ultraproduct embeddings and amenability]{On ultraproduct embeddings and amenability for tracial von Neumann algebras}

\author{Scott Atkinson and Srivatsav Kunnawalkam Elayavalli}

\address{University of California Riverside, Riverside, CA, USA}
\email{scott.atkinson@ucr.edu}
\address{Vanderbilt University, Nashville, TN, USA}
\email{srivatsav.kunnawalkam.elayavalli@vanderbilt.edu}

\begin{abstract}
We define the notion of \emph{self-tracial stability} for tracial von Neumann algebras and show that a tracial von Neumann algebra satisfying the Connes Embedding Problem is self-tracially stable if and only if it is amenable. We then generalize a result of Jung by showing that a separable tracial von Neumann algebra that satisfies the Connes Embedding Problem is amenable if and only if any two embeddings into $R^\mc{U}$ are ucp-conjugate. Moreover we show that for a II$_1$ factor $N$ satisfying CEP, the space $\mb{H}$om$(N, \prod_{k\to \mc{U}}M_k)$ of unitary equivalence classes of embeddings is separable if and only $N$ is hyperfinite. This resolves a question of Popa for Connes embeddable factors. These results hold when we further ask that the pairs of embeddings commute, admitting a nontrivial action of $\text{Out}(N\otimes N)$ on $\HOM(N\otimes N, \prod_{k\to \mc{U}}M_k)$ whenever  $N$ is non-amenable.  We also obtain an analogous result for commuting sofic representations of countable sofic groups.
\end{abstract}

\maketitle

%\section{Introduction}

The problem of identifying amenability in von Neumann algebras has received an immense amount of attention going back to the beginning of the subject.  Indeed, the first result of this kind was Murray-von Neumann's famous proof that the hyperfinite II$_1$-factor $R$ is not isomorphic to the free group factor $L(\mb{F}_2)$ using Property $\Gamma$ (see \S6 of \cite{mvn4}).  Likewise, the subject of ultraproduct analysis for tracial von Neumann algebras has a deep and rich history dating back to the early years of the subject with McDuff's article \cite{mcduff} making the value of such techniques readily apparent.  Connes's seminal paper \cite{connes} makes use of such ultraproduct considerations to establish the groundbreaking classification of injective von Neumann algebras.  Throughout the history of the subject, problems in von Neumann algebras involving amenability and ultraproducts have steadily maintained momentum, interest, and relevance.

In this article, we identify multiple new characterizations of amenability for tracial von Neumann algebras in the context of embeddings into ultraproducts. These results are obtained under the assumption that the algebra in question satisfies the Connes Embedding Problem (CEP):

\begin{prob*}[CEP]
Let $N$ be a separably acting tracial von Neumann algebra, $R$ be the separably acting amenable II$_1$-factor, and $\mc{U}$ be a free ultrafilter on $\mb{N}$. Does there exist an embedding $N \rightarrow R^\mc{U}$? 
\end{prob*}

\noin An affirmative resolution of this problem was famously conjectured in \cite{connes}. Recently, in \cite{connessol} Ji-Natarajan-Vidick-Wright resolved CEP in the negative.  This means that our assumption that a tracial von Neumann algebra satisfies CEP is nontrivial.

For the first characterization of this paper, we introduce the idea of \emph{self-tracial stability} for tracial von Neumann algebras together with related notions.  A tracial von Neumann algebra $(N,\tau)$ is self-tracially stable if maps $N \rightarrow N$ that are approximate $*$-homomorphisms are near honest $*$-homomorphisms; self-tracial stability is a specific instance of tracial stability---see Definition \ref{ssdef}. It is a direct consequence of a well-known result that amenability implies self-tracial stability. We prove the following theorem.

\begin{repthm}{tscon} 
Let $(N,\tau)$ be a separably acting tracial von Neumann algebra satisfying CEP. Then $N$ is self-tracially stable if and only if $N$ is amenable.
\end{repthm}

\noin This resolves Question 1.8 of \cite{sellin} modulo CEP.  It follows that self-tracial stability is not axiomatizable (see Proposition \ref{axiom}). One appeal of self-tracial stability is that its definition does not require a CEP assumption.  Thus when we perceive this characterization as a necessary condition of CEP we obtain a strategy for identifying counterexamples of CEP by way of contraposition: if $N$ is self-tracially stable and not amenable, then $N$ does not satisfy CEP. 

The next characterization generalizes a well-known result of Jung.  In \cite{jung}, Jung showed that if $(N,\tau)$ is a finitely generated tracial von Neumann algebra satisfying CEP, then $N$ is amenable if and only if any pair of embeddings into $R^\mc{U}$ are unitarily conjugate. We will call this result \emph{Jung's theorem}.  For the same reason as in the previous paragraph, the \tql only if\tqr direction of Jung's theorem is a direct exercise.  Jung's strategy utilizes the finitary notion of \emph{(quasi)tubularity}--a condition on the unitary orbits of matrices that tracially approximate the generators of $N$. (It is important to note that Hadwin's \emph{dimension ratio} for the generators of a tracial von Neumann algebra is a quantified predecessor of tubularity--see item (6) of Theorem 3.3 in \cite{hadwin}.)
The concept of (quasi)tubularity is the mechanism one uses to pivot from unitary equivalence of embeddings to semidiscreteness (i.e. amenability). The generalization appearing in this article is based on the notion of ucp-conjugation. Two embeddings $\pi, \rho: N\to R^{\mc{U}}$ are ucp-conjugate if there exists a sequence of unital completely positive maps $\varphi_k: R\to R$ that assemble to conjugate $\pi$ and $\rho$ (see Definition \ref{wms}). The result is as follows.

\begin{repthm}{jungv2} 
Let $(N,\tau)$ be a separably acting tracial von Neumann algebra satisfying CEP. Then any pair of embeddings of $N$ into $R^\mc{U}$ are ucp-conjugate if and only if $N$ is injective.  
\end{repthm} 

\noin Note that we have done away with the assumption that $N$ is finitely generated.  To prove the above theorem, we adapt Jung's tubularity for the ucp setting, calling the concept \emph{complete tubularity} (see Definition \ref{tube}).  In the process of defining complete tubularity, we take care to expand its scope to apply to infinitely many generators. Complete tubularity in conjunction with an unpublished characterization of injectivity by Kishimoto (see Proposition \ref{kishimoto}) serves as the link from ucp-conjugacy of embeddings to injectivity.

This characterization of injectivity in terms of ucp-conjugacy sheds light on a question of Popa  appearing in the third paragraph of \S\S2.3 of \cite{popa}  regarding unitary equivalence classes of embeddings of tracial von Neumann algebras into an arbitrary ultraproduct of II$_1$-factors $\ds \prod_{k\rightarrow \mc{U}}M_k$. Popa's question is as follows.

\begin{b?}[\cite{popa}]
If $N$ is a separable von Neumann subalgebra of an ultraproduct II$_1$-factor $\ds\prod_{k\rightarrow \mc{U}}M_k$, then how large is the space of embeddings of $N$ into $\ds\prod_{k\rightarrow \mc{U}}M_k$ modulo unitary equivalence?
\end{b?}

\noin In the appendix of \cite{topdyn}, Ozawa showed that the space of unitary equivalence classes of embeddings of a separably acting II$_1$-factor $N$ into $R^\mc{U}$ is separable if and only if $N$ is amenable. Evidently, this result served as inspiration for Popa's question. One must take care in considering the general case of an ultraproduct codomain $\ds\prod_{k\rightarrow \mc{U}}M_k$. The obstruction is that there are more unitaries available in $\ds\prod_{k\rightarrow \mc{U}}M_k$ that could \emph{a priori} coarsen the space of unitary equivalence classes. As an immediate consequence of Theorem \ref{jungv2}, we obtain the following:

\begin{repcor}{jungv2cor} 
Let $(N,\tau)$ be a non-amenable separably acting tracial von Neumann algebra satisfying CEP.  Then there are at least two distinct embeddings of $N$ into an arbitrary ultraproduct of II$_1$-factors $\ds \prod_{k\rightarrow \mc{U}}M_k$ up to unitary conjugation.
\end{repcor} 

\noin This is a first step in answering Popa's question in full generality. To completely resolve Popa's question, we apply techniques from \cite{saa} and obtain the following theorem.

\begin{repthm}{mcd}
Given a separably acting finite von Neumann algebra $N$ satisfying CEP, the space of unitary equivalence classes of embeddings of $N$ into an ultraproduct of II$_1$-factors $\ds\prod_{k\rightarrow \mc{U}}M_k$ is separable if and only if $N$ is amenable.
\end{repthm}

Using the self-absorbing nature of $R$, we extend the above results to the setting where we further insist that the embeddings commute (Theorem \ref{mr}).  This result has applications to topics including weak approximate unitary equivalence of commuting embeddings into separable factors, the action of $\text{Out}(N\otimes N)$ on the space $\HOM(N\otimes N, \prod_{k\to \mc{U}} M_k)$, and the so-called ultra ucp lifting property (see \S\S\ref{comem}).  We also obtain an analogous result for commuting sofic representations of countable sofic groups.  While there is no self-absorbing behavior available to exploit in the sofic setting, we show how the techniques of Elek-Szabo from \cite{elekszabo} in their proof of Jung's theorem for sofic groups apply to produce the desired result (see \S\S\ref{comsof}). 
%The consideration of amplifications is a bit more subtle in the sofic group case, than in the von Neumann algebra case. 
%However, we are able to easily adapt the techniques of Elek and Szabo in \cite{elekszabo} to produce the desired result.

We conclude with comments and questions regarding some notions from the paper in the group setting. With its official debut in Hadwin--Shulman's 2018 article (2016 preprint) \cite{hadshu}, tracial stability is a fairly young topic in operator algebras. Analogous notions of stability in group theory have a longer history.  Given a class of groups $\mc{G}$ equipped with bi-invariant metrics, a group $G$ is \emph{$\mc{G}$-stable} if approximate homomorphisms of $G$ into members of $\mc{G}$ are near honest homomorphisms (see Definition \ref{gpstab}). Some examples of a class $\mc{G}$ commonly considered include \[\text{HS}:= \left\{\mc{U}(\mb{M}_n), \text{normalized Hilbert-Schmidt norm}\right\}_{n \in \mb{N}}\] and \[\mc{P}:= \left\{ n\times n \text{ permutation matrices, Hamming metric}\right\}_{n\in \mb{N}}.\] Thom's ICM survey \cite{thom} discusses recent results, applications, and questions related to stability for groups. Some other notable group theoretic references on this topic include  \cite{arzpau, hadshu2, deglluth, beluth, beclub}.

\subsection*{Acknowledgments} The authors gratefully acknowledge D. Bisch, I. Farah, A. Levit, L. P\u{a}unescu, J. Peterson, and P. Spaas for helpful conversations about these results.  We also thank B. Hayes, B. Nelson, and D. Sherman for useful discussions and careful proofreading. Special gratitude is due to A. Ioana and N. Ozawa.

\section{Preliminaries}\label{prelim}

\subsection{Amenability}\label{amen}

We begin this section with a discussion of amenability for von Neumann algebras.  
%In particular we will discuss some well-known as well as some obscure properties and characterizations of amenability for von Neumann algebras.

\begin{dfn}\label{amendef}\hspace*{\fill}

\begin{enumerate}

\item A von Neumann algebra $N$ is \emph{amenable} if any derivation $\delta$ of $N$ into a dual Banach bimodule $X$ is inner.  That is, any map $\delta: N \rightarrow X$ satisfying $\delta(ab) = a\delta(b) + \delta(a)b$ for every $a,b \in N$ must be of the form $\delta(a) = ax-xa$ for some $x \in X$.

\item  A von Neumann algebra $N$ is \emph{hyperfinite} if it can be expressed as the $\sigma$-weak closure of an increasing union of finite-dimensional subalgebras.

\item A von Neumann algebra $N$ is \emph{injective} if for any inclusion $X \subset Y$ of operator systems and ucp map $\varphi: X \rightarrow N$, there exists a ucp map $\tilde{\varphi}: Y \rightarrow N$ such that $\tilde{\varphi}|_X = \varphi$.

\item A von Neumann algebra $N$ is \emph{semidiscrete} if there is a pair of nets of ucp maps $\varphi_\alpha: N \rightarrow \mb{M}_{n(\alpha)}, \psi_\alpha: \mb{M}_{n(\alpha)} \rightarrow N$ such that $\psi_\alpha \circ \varphi_\alpha \rightarrow \text{id}_N$ in the point-ultraweak topology.

\end{enumerate}

\end{dfn}

\noin It is well-known that the four conditions in Definition \ref{amendef} are equivalent.  This equivalence is largely due to the results from \cite{connes}, but we should also mention \cite{choeff,wass,newproof} when discussing this result.  There are more conditions well-known to be equivalent to the above four, but for the purposes of this article, we only consider these four. When discussing results regarding amenability/hyperfiniteness/injectivity/\linebreak semidiscreteness we will use the most relevant term; for instance, if an argument's conclusion is that an algebra satisfies the conditions for injectivity, we will call the algebra injective.  When the context is general, we will use the term \emph{amenable} as the umbrella term to refer to an algebra satisfying these equivalent conditions.  

\begin{exmpl}
The most well-known example of an amenable von Neumann algebra is the separable hyperfinite II$_1$-factor, denoted by $R$.  Murray-von Neumann showed in \cite{mvn4} that $R$ is the unique separable hyperfinite II$_1$-factor.  To sketch a construction, consider the infinite tensor product $\ds \bigotimes_\mb{N} \mb{M}_2$.  Using the unique tracial state, form a GNS representation of $\ds \bigotimes_\mb{N} \mb{M}_2$ and take the bicommutant.
\end{exmpl}

%Amenability has many useful alternative characterizations--the collection of which the present article expands.  
%
%
%
%\noin In \cite{connes}, Connes showed (in most cases) that a von Neumann algebra is hyperfinite if and only if it is injective in the category with operator systems as objects and unital completely positive (ucp) maps as morphisms.  Haagerup settled the remaining case in \cite{newproof}.
%
%
%
%\noin The results in \cite{connes} and \cite{newproof} yield the equivalences between semidiscreteness, injectivity, and hyperfiniteness.  

This paper focuses on separably acting tracial von Neumann algebras.    A \emph{tracial} von Neumann algebra is given by a pair $(N,\tau)$ where $N$ is a von Neumann algebra and $\tau$ is a faithful normal tracial state on $N$.  Given two tracial von Neumann algebras $(N,\tau), (M,\sigma)$, an \emph{embedding} of $(N,\tau)$ into $(M,\sigma)$ is an injective unital $*$-homomorphism $\pi: (N,\tau) \rightarrow (M, \sigma)$ such that $\sigma \circ \pi = \tau$.  Given a tracial von Neumann algebra $(N,\tau)$, we often consider the \emph{trace-norm} $||\cdot||_2$ on $N$ induced by $\tau$ defined as follows.  For $x \in N, ||x||_2 = \tau(x^*x)^\frac{1}{2}$.  Given $\varepsilon > 0$ and two subsets $A,B \subset N$, the notation $A \subset_{||\cdot||_2,\varepsilon} B$ denotes the condition that for every $a \in A$ there is a $b \in B$ such that $||a-b||_2 < \varepsilon$.
%It is a well-known fact that given a separably acting finite von Neumann algebra $M$ with distinguished tracial state $\tau$, there is a separably acting II$_1$-factor $N$ with unique tracial state $\tau_N$ such that $M$ is unitally contained in $N$ and $\tau_N  |_M = \tau$.  In this situation we say that $(M,\tau)$ is \emph{embedded} in $N$.

We now present some lesser-known characterizations of amenability for separably acting tracial von Neumann algebras. The following proposition from \cite{jung} provides a useful finitary viewpoint of semidiscreteness:
\begin{prop}[\cite{jung}]\label{jungslem}
Let $(N,\tau)$ be a tracial von Neumann algebra.  We have that $N$ is semidiscrete if and only if for any finite subset $F \subset N$ and any $\varepsilon > 0$, there exist a tracial von Neumann algebra $(N_\varepsilon,\tau_\varepsilon)$, an embedding $\pi_\varepsilon: (N,\tau) \hookrightarrow (N_\varepsilon,\tau_\varepsilon)$, and a unital finite-dimensional subalgebra $A_\varepsilon \subset N_\varepsilon$ such that \[\pi_\varepsilon(F) \subset_{||\cdot||_2, \varepsilon} A_\varepsilon.\]
\end{prop}
\noin Note that the version proved in \cite{jung} assumes that $F$ is a generating set, but after a reduction, Jung's proof establishes the above statement.  Thus we do not rely on an assumption of $N$ being finitely generated.

The following proposition is the crucial step in the proof of Theorem \ref{jungv2}. 
%It is a strong characterization of injectivity. 
The result in Proposition \ref{kishimoto} is not new. It is an unpublished result of Kishimoto. The proof recorded here was communicated to us by N. Ozawa to whom we extend our sincere gratitude.

\begin{prop}[Kishimoto]\label{kishimoto}
Let $(N,\tau)$ be a tracial von Neumann algebra.  We have that $N$ is injective if and only if for any finite subset $F \subset (N)_{\leq 1}$ and any $\varepsilon > 0$, there exist $J \in \mb{N}$ and a ucp map $\rho: \mb{M}_J \rightarrow N$ such that \[F \subset_{||\cdot||_2, \varepsilon} \rho((\mb{M}_J)_{\leq 1}).\]
\end{prop}

\noin The idea of the proof is to construct an averaging map which will serve as a conditional expectation from $B(L^2(N,\tau))$ to $N'$, thereby showing that $N'$ is injective, and hence $N$ is injective. Before proceeding with the proof of Proposition \ref{kishimoto}, we need the following lemma.

\begin{lem}\label{Klem}
Let $(N,\tau)$ be a tracial von Neumann algebra, and fix $\varepsilon > 0$. Given a ucp map $\rho: \mb{M}_J \rightarrow N$, a unitary $u \in \mc{U}(N)$, and a contraction $x \in (\mb{M}_J)_{\leq 1}$ such that $||u - \rho(x)||_2 < \varepsilon$.  For every unitary $v_0 \in \mb{M}_J$ for which $x = v_0|x|$, we have $||u- \rho(v_0)||_2 < 3\varepsilon^\frac{1}{2}$.
\end{lem}

\begin{proof}
Suppose that $0 < \varepsilon < 2^{-1}$.  Since $\rho$ is 2-positive and $\tau$ is positive, we have that \[\tau(\rho(x^*x)) \geq \tau(\rho(x^*)\rho(x)) > 1- 2\varepsilon\] by  Schwarz's inequality and the reverse triangle inequality.  Observe that by \linebreak Schwarz's inequality and functional calculus
\begin{align*}
||\rho(v_0) -\rho(x)||_2^2 &=  ||\rho(v_0(1-|x|))||_2^2\\
 & = \tau(\rho(v_0(1-|x|))^*\rho(v_0(1-|x|)))\\
&\leq \tau(\rho((1-|x|)v_0^*v_0(1-|x|)) \\
%&= \tau(\rho(v_0^*v_0 - 2|x| + |x|^2))\\
%&= \tau(\rho(1 - 2|x| + |x|^2))\\
&= \tau(\rho((1-|x|)^2))\\
&\leq \tau(\rho(1-|x|^2))\\
&< 2\varepsilon.
\end{align*}
Thus we have \[||\rho(v_0) - u||_2 \leq ||\rho(v_0) - \rho(x)||_2 + ||\rho(x) - u||_2 < (2\varepsilon)^\frac{1}{2} + \varepsilon < 3\varepsilon^\frac{1}{2}.\qedhere\]
\end{proof}

\begin{proof}[Proof of Proposition \ref{kishimoto}]
Consider $N \subset B(L^2(N,\tau))$.  It suffices to show that $N'$ is injective.

Fix a finite subset $F \subset \mc{U}(N) \subset (N)_{\leq 1}$ and $\varepsilon > 0$.  Let $\rho: \mb{M}_J \rightarrow N$ be such that \[F \subset_{||\cdot||_2, \varepsilon} \rho((\mb{M}_{J})_{\leq 1}).\]  Then by Lemma \ref{Klem} we have that  for every $u \in F$ there is a $v_0 \in \mc{U}(\mb{M}_{J})$ such that $||u - \rho(v_0)||_2 \leq 3\varepsilon^\frac{1}{2}$.  By Stinespring's theorem there is a $*$-homomorphism $\pi: \mb{M}_{J} \rightarrow \mb{M}_L(N)$ such that $\rho(x) = W^*\pi(x)W$ where $W = (1,0,\dots,0)^T \in \mb{M}_{L,1}$.  

Let $\xi_\tau$ denote the cyclic vector for $L^2(N,\tau)$ corresponding to $1 \in N$, and let $\mc{K}$ denote the super-Hilbert space containing $L^2(N,\tau)$ afforded by the above Stinespring dilation. Note that for $u \in \mc{U}(N), v \in \mc{U}(\mb{M}_{J}),$ and $b' \in N'$ we have 
\begin{align*}
Wub' &= (b' \otimes I_L)Wu\\
&\text{and}\\
\pi(v)Wb' &= (b'\otimes I_L)\pi(v)W.
\end{align*}
So 
\begin{align*}
||(Wu-\pi(v)W)b'\xi_\tau||_\mc{K}^2 &\leq ||b'||^2||(Wu-\pi(v)W)\xi_\tau||^2_\mc{K}\\
& \leq ||b'||^2 (2 - 2 \FR{Re}\tau(u^*\rho(v)))\\
& \leq 2||b'||^2 ||u - \rho(v)||_2.
\end{align*}
The last inequality is achieved by factoring out 2, and writing 1 as $\FR{Re}\tau(u^*u)$, and then applying Cauchy-Schwarz inequality. Hence for $v_0 \in \mc{U}(\mb{M}_{J})$ such that $||u-\rho(v_0)||_2 < 3\varepsilon^\frac{1}{2}$, we have 
\begin{align}
||(Wu-\pi(v_0)W)b'\xi_\tau||_\mc{K}^2 < 2||b'||^2(3\varepsilon^{1/2}) \label{1}
\end{align} for every $u \in F$.

Next we construct a ucp map $\theta: B(L^2(N,\tau)) \rightarrow B(L^2(N,\tau))$ in the following way: \[ \theta(x):= \int_{\mc{U}(\mb{M}_{J})} W^*\pi(v^*)(x\otimes 1_L) \pi(v)Wdv.\]
Here, integration is with respect to the Haar probability measure on $\mc{U}(\mb{M}_{J})$.  Given $b' \in N'$, we see that 
\begin{align*}
\theta(b') &= \int_{\mc{U}(\mb{M}_{J})} W^*\pi(v^*)(b'\otimes 1_L) \pi(v)Wdv\\
&= b' \int_{\mc{U}(\mb{M}_{J})} W^*\pi(v^*)(I_{B(L^2(N,\tau))}\otimes 1_L) \pi(v)Wdv\\
&=b'.
\end{align*}
Thus $\theta|_{N'} = \text{id}_{N'}$. Let $v_0 \in \mc{U}(\mb{M}_J)$ be as above, and let $u \in F, x \in B(L^2(N,\tau))_{\leq 1},$ and $b', b'' \in N'$ be given.  First note that by the translation invariance of the Haar measure, we have that 

\begin{align*}
&\int_{\mc{U}(\mb{M}_J)} \left\langle u^*W^*\pi(v^*)(x \otimes 1_L)\pi(v)\pi(v_0)Wb'\xi_\tau | b''\xi_\tau\right\rangle_\tau dv\\
=& \int_{\mc{U}(\mb{M}_J)} \left\langle u^*W^*\pi(v_0)\pi(v^*)(x \otimes 1_L)\pi(v)Wb'\xi_\tau | b''\xi_\tau\right\rangle_\tau dv.
\end{align*}
Next, we see that by \eqref{1}

\[\int_{\mc{U}(\mb{M}_J)} \left| \left\langle\left(u^*W^*\pi(v^*)(x\otimes 1_L) \pi(v)Wu - u^*W^*\pi(v^*)(x\otimes 1_L) \pi(v)\pi(v_0)W\right)b'\xi_\tau \Big| b'' \xi_\tau\right\rangle_\tau\right|dv\]
\begin{align*}
=& \int_{\mc{U}(\mb{M}_J)} \left| \left\langle  u^*W^*\pi(v^*)(x\otimes 1_L) \pi(v)(Wu - \pi(v_0)W)b'\xi_\tau \Big| b'' \xi_\tau\right\rangle_\tau\right|dv\\
<& \sqrt{6\varepsilon^{1/2}} ||b'|| \cdot ||b''||.
\end{align*}
Also, by \eqref{1},

\[\int_{\mc{U}(\mb{M}_J)} \left|\left\langle \left(u^*W^*\pi(v_0)\pi(v^*)(x \otimes 1_L)\pi(v)W - W^*\pi(v^*)(x\otimes 1_L)\pi(v)W\right)b'\xi_\tau | b''\xi_\tau\right\rangle_\tau\right| dv\]
 \begin{align*}
 =& \int_{\mc{U}(\mb{M}_J)} \left|\left\langle \pi(v_0)\pi(v^*)(x \otimes 1_L)\pi(v)Wb'\xi_\tau | (Wu - \pi(v_0)W) b''\xi_\tau\right\rangle_\tau\right| dv\\
 <& \sqrt{6\varepsilon^{1/2}} ||b'|| \cdot ||b''||. 
\end{align*}
Thus it follows that \[|\langle (u^*\theta(x)u - \theta(x))b'\xi_\tau | b''\xi_\tau\rangle| < 2\sqrt{6\varepsilon^{1/2}}||b'||\cdot ||b''||.\]

%\begin{align*}
%&|\langle (u^*\theta(x)u - \theta(x)) b'\xi_\tau | b''\xi_\tau\rangle_\tau|\\
% & = \left| \left\langle\left(\int_{\mc{U}(\mb{M}_J)} (u^*W^*\pi(v^*)(x\otimes 1_L) \pi(v)Wu - W^*\pi(v^*)(x\otimes 1_L) \pi(v)W)dv\right)b'\xi_\tau \Big| b'' \xi_\tau\right\rangle_\tau\right|\\
% &\leq \int_{\mc{U}(\mb{M}_J)} \left| \left\langle\left(u^*W^*\pi(v^*)(x\otimes 1_L) \pi(v)Wu - W^*\pi(v^*)(x\otimes 1_L) \pi(v)W\right)b'\xi_\tau \Big| b'' \xi_\tau\right\rangle_\tau\right|dv\\
% &\leq \int_{\mc{U}(\mb{M}_J)} \left( \left| \left\langle\left(u^*W^*\pi(v^*)(x\otimes 1_L) \pi(v)Wu - u^*W^*\pi(v^*)(x\otimes 1_L) \pi(v)\pi(v_0)W\right)b'\xi_\tau \Big| b'' \xi_\tau\right\rangle_\tau\right|\right.\\ 
% & + \left.\left|\left\langle\left( u^*W^*\pi(v^*)(x\otimes 1_L) \pi(v)\pi(v_0)W -W^*\pi(v^*)(x\otimes 1_L) \pi(v)W\right)b'\xi_\tau\Big| b'' \xi_\tau\right\rangle_\tau\right|\right)dv
%\end{align*}

Let $\theta_{(F,\varepsilon)}$ denote the ucp map constructed above for given finite $F \subset \mc{U}(N)$ and $\varepsilon > 0$.  Then $\left\{\theta_{(F,\varepsilon)}\right\}$ forms a net of ucp maps on $B(L^2(N,\tau))$.  By Theorem 1.3.7 of \cite{brownozawa}, there is a limit point $\theta_0$ of this net in the point-ultraweak topology.  Thus for any $u \in \mc{U}(N)$ and $x \in B(L^2(N,\tau)))$ by the above estimates, $u^*\theta_0(x)u = \theta_0(x)$.  Hence $\theta_0(B(L^2(N,\tau))) \subset N'$.  Since $\theta_0$ is contractive, it follows that \[\theta_0: B(L^2(N,\tau)) \rightarrow B(L^2(N,\tau))\] is a conditional expectation onto $N'$.  So $N'$ is injective.
\end{proof}

The next proposition is relevant to the narrative of this paper, and its proof is straightforward. Before presenting the proposition, a definition is in order.

\begin{dfn}\label{wauedef}
Given a separable von Neumann algebra $N$ and a separable II$_1$-factor $M$, two unital $*$-homomorphisms $\pi, \rho: N \rightarrow M$ are \emph{weakly approximately unitarily equivalent} ---denoted $\pi \sim_\text{waue} \rho$---if there is a sequence of unitaries $\left\{ u_n \right\} \subset \mathcal{U}(M)$ such that for every $x \in N$, \[\lim_{n\rightarrow \infty} ||\pi(x) - u_n \rho(x) u_n^*||_2 = 0\]  (see \cite{voiculescu}, \cite{ding-hadwin}, and \cite{orbits}).  Alternatively, $\pi \sim_\text{waue}\rho$ if and only if for any finite subset $F \subset N$ and $\varepsilon >0$, there exists a unitary $u \in \mc{U}(M)$ such that $||\pi(x) - u\rho(x)u^*||_2< \varepsilon$ for every $x \in F$.
\end{dfn}

\begin{prop}\label{waue}
Let $N$ be a II$_1$-factor, and let $\pi, \rho: R \rightarrow N$ be two unital embeddings.    Then $\pi\sim_\text{waue}\rho$. \end{prop}

% \begin{proof}
% Fix a finite subset $F \subset R$ and $\varepsilon >0$. Since $\ds R = \overline{\otimes_\mb{N} \mb{M}_2}$, there is a matrix subalgebra $\mb{M}_k \subset R$ such that for every $x \in F$ there is a $b \in \mb{M}_k$ such that $\ds ||x - b||_2 < \frac{\varepsilon}{2}$.  Consider $\pi|_{\mb{M}_k}$ and $\rho|_{\mb{M}_k}$.  It is a direct exercise in exploiting Murray-von Neumann equivalence to show that any two unital embeddings of a matrix algebra into a II$_1$-factor are unitarily equivalent.  Thus, there exists a unitary $u \in \mc{U}(N)$ such that $\pi(b) = u\rho(b)u^*$ for every $b \in \mb{M}_k$.  The claim follows.
% \end{proof}

\noin It turns out that this property characterizes amenability (after one makes some necessary assumptions)---see Theorem \ref{saasep}.

\subsection{Ultraproducts}
We next recall the tracial ultraproduct construction.  Let $\mc{U}$ denote a free ultrafilter on $\mb{N}$ (see Appendix A of \cite{brownozawa} for a definition).  For each $k \in \mb{N}$, let $\mc{A}_k$ be a unital $C^*$-algebra with tracial state $\tau_k$, and let $||\cdot ||_{2,k}$ denote the induced trace-seminorm. Define the sequence space \[\prod_{k \in \mb{N}}^\infty \mc{A}_k:= \left\{(a_k)_{k \in \mb{N}}: a_k \in \mc{A}_k \text{ and } \sup_k ||a_k|| < \infty\right\}.\] If $\mc{A}_k = \mc{A}$ for every $k \in \mb{N}$ we write $\ell^\infty(\mc{A})$ for $\ds\prod_{k \in \mb{N}}^\infty \mc{A}_k$. We define the \emph{tracial ultraproduct} of the $(\mc{A}_k,\tau_k)$'s, denoted $\ds \prod_{k\rightarrow \mc{U}}(\mc{A}_k,\tau_k)$ (or simply $\ds\prod_{k\rightarrow \mc{U}}\mc{A}_k$ when the context is clear), to be given by \[\prod_{k\rightarrow \mc{U}}(\mc{A}_k,\tau_k) := \left(\prod_{k \in \mb{N}}^\infty \mc{A}_k \right)\Big/ \mc{I}_\mc{U}\] where \[\ds \mc{I}_\mc{U} := \left\{ (a_k) \in \prod_{k \in \mb{N}}^\infty \mc{A}_k : \lim_{k \rightarrow \mc{U}} ||a_k||_{2,k} = 0\right\}.\] Given a sequence $\ds(a_k) \in \prod_{k \in\mb{N}}^\infty \mc{A}_k$, let $(a_k)_\mc{U}$ denote the coset of $(a_k)$ in $\ds\prod_{k\rightarrow \mc{U}}(\mc{A}_k,\tau_k)$.

%In this article the typical situation will be that $\mc{A}_k = M$ for all $k \in \mb{N}$ where $M$ is a separably acting II$_1$-factor.  In this case, there is a unique tracial state, so the tracial ultraproduct will be denoted $M^\mc{U}$, and we call $M^\mc{U}$ the \emph{tracial ultrapower} of $M$.  Since we are working exclusively in the tracial setting, we will simply write \tql ultrapower\tqr instead of \tql tracial ultrapower.\tqr

It will be useful to set the following notation.  Let $(N,\tau)$ be a tracial von Neumann algebra and for each $k \in \mb{N}$ let $(M_k,\tau_k)$ be a tracial von Neumann algebra.  Given an embedding $\ds \pi: N \rightarrow \prod_{k\rightarrow \mc{U}}M_k$, for each $k \in \mb{N}$ let $\pi(x)_k \in M_k$ be so that $\pi(x) = (\pi(x)_k)_\mc{U}$.  That is, $(\pi(x)_k)_{k \in \mb{N}}$ is a representative from $\ds\prod_\mb{N} M_k$ of the coset of $\pi(x)$.

\subsection{Microstates}
Microstates were introduced by Voiculescu in paper II of his revolutionary series of papers titled \tql On the Analogues of Entropy and of Fisher's Information Measure in Free Probability Theory I-VI\tqr \cite{entropy1,entropy,entropy3,entropy4,entropy5,entropy6}. Roughly speaking, a microstate is a tuple of matrices that approximates a given tuple of operators in many moments. The consideration of microstates helps to make embeddings into $R^{\mc{U}}$ more tractable. It is well known that satisfying CEP is equivalent to \tql having sufficiently many microstates.\tqr 
%Realizing the microstates spaces as subsets of Euclidean space with appropriate norm, Voiculescu studied the normalized asymptotic logarithmic volume and Minkowski dimensions of these to build a theory of non commutative entropy. This free entropy theory has been used to solve difficult open problems in structure theory of von Neumann algebras and free group factors (see the survey \cite{freeentropy} for an account of applications). For the purposes of our paper though, we only need the following definition and propositions.

\begin{dfn}
Let $(N,\tau)$ be a tracial von Neumann algebra, and let $X=(x_1,\hdots, x_n)$ be a finite set of operators in $(N)_{\leq1}$. Fix $k \in \mb{N}$, and as usual let $\mb{M}_k$ denote the algebra of $k \times k$ matrices with complex entries. Write $\xi=(\xi_1,\hdots, \xi_n)\in \mb{M}_k^n$. Define the \emph{microstate space} $\Gamma_1(X;d,k,\gamma)$  to be the set of all $\xi \in (\mb{M}_k)_{\leq 1}^n$ such that  $|\text{tr}_k(p(\xi_1,\dots,\xi_n))-\tau(p(x_1,\dots,x_n))| \leq \gamma$ for every noncommutative $*$-monomial $p$ in $n$ variables of degree at most $d$. Here $\text{tr}_k$ denotes the normalized trace on the matrix algebra $\mb{M}_k$. The set $\Gamma(X;d,k,\gamma)$ is defined identically without the operator norm constraint on the microstates.
\end{dfn}

The following proposition follows from a direct normalization argument.

\begin{prop} Let $(N,\tau)$ be a tracial von Neumann algebra satisfying CEP, and let $X=(x_1,\hdots, x_n)\subset (N)_{\leq 1}$ be such that $W^*(x_1,\dots, x_n) = N$. Then, for every $k\in \mb{N}$ there exists an $N_k\in \mb{N}$ such that $\Gamma_1(X;k,N_k,k^{-1})\neq \emptyset$. In particular, for a fixed sequence of microstates $\xi^{(k)}=(m_1^{(k)},\dots, m_n^{(k)})\in \Gamma_1(X;k,N_k,k^{-1})$, there exists an embedding $\sigma: N\to R^{\mc{U}}$ such that $\sigma(x_j)= \left(\xi_j^{(k)}\right)_\mc{U}$ for $1 \leq j \leq n$.
\end{prop}

%\begin{proof}
%From definitions, it is easy to see that for $d\in \mb{N}$ there exists an infinite sequence of microstates $m_r\in \Gamma(d_r, N_r, d_r^{-1})$ such that $N_r$ and $d_r$ are increasing and $\lim_{r\to \mc{U}}||m_r||\leq 1$. Now, find $\varepsilon>0$ such that $(1-\varepsilon)||m_r||\leq1$ for sufficiently large $r$, and also $| (1-\varepsilon)^dtr_{N_r}(p(m))-\tau(p(X))|<d^{-1}$ for every non-commutative $*$- monomial of degree at most $d$ . This gives the desired microstate, namely, $(1-\varepsilon)m_r$ for $r$ sufficiently large.
%\end{proof}

Typically microstates are used in conjunction with finitely generated von Neumann algebras.  We now discuss a method to use microstates to obtain an embedding of a separably acting tracial von Neumann algebra.  We first recall the following lemma.

\begin{lem}\label{tracepreserve}
Let $(N,\tau)$ and $(M,\sigma)$ be tracial von Neumann algebras. If $A\subset N$ is a $*$-subalgebra with $W^*(A) = N$, then a trace-preserving $*$-homomorphism $\pi: A \rightarrow M$ extends to an embedding $\tilde{\pi}: N \rightarrow M$.
\end{lem}

Now let $(N,\tau)$ be a separable tracial von Neumann algebra satisfying CEP, and let $X:=\left\{x_j\right\}_{j \in\mb{N}}\subset (N)_{\leq 1}$ be a generating set for $N$ ($\left\{x_j\right\}_{j \in \mb{N}}$ can and often will be taken to be a countable $||\cdot||_2$-dense subset of $(N)_{\leq 1}$).  For $k \in \mb{N}$ let $X_k:= \left\{x_1,\dots, x_k\right\}$. As discussed above, for each $k \in \mb{N}$ there is a $J_k \in \mb{N}$ for which $\Gamma_1(X_k; k, J_k, k^{-1})$ is nonempty.  So for each $k \in \mb{N}$ let \[\xi^{(k)} = (m^{(k)}_1,\dots, m^{(k)}_k) \in \Gamma_1(X_k; k, J_k, k^{-1}).\]  Let $A \subset N$ denote the $*$-subalgebra generated by $X$.  Consider $\pi: A \rightarrow R^\mc{U}$ defined on $X$ by $\pi(x_j) = (\xi^{(k)}_j)_\mc{U}$ for each $j \in \mb{N}$ (where $\xi^{(k)}_j = 0$ if $j > k$).  Then $\pi$ preserves the trace on $A$, and by Lemma \ref{tracepreserve}, $\pi$ extends to an embedding $\tilde{\pi}: N \rightarrow R^\mc{U}$.

\section{Self-tracial stability}\label{ss}

\begin{dfn}[\cite{hadshu}]\label{ssdef}
Let $\ms{C}$-denote a class of $C^*$-algebras closed under \linebreak $*$-isomorphism.  A separably acting von Neumann algebra $N$ is \emph{$\ms{C}$-tracially stable} if for any unital $*$-homomorphism $\pi: N \rightarrow (\mc{A}_k,\tau_k)^\mc{U}$ with $\mc{A}_k \in \ms{C}$ and $\tau_k$ a tracial state on $\mc{A}_k$, there exist unital $*$-homomorphisms $\pi_k: N \rightarrow \mc{A}_k$ such that $\pi(x) = (\pi_k(x))_\mc{U}$ for every $x \in N$. In this case, we say $\pi$ \emph{lifts}. A separably acting tracial von Neumann algebra $(N,\tau)$ is \emph{self-tracially stable} if $N$ is $\left\{(N,\tau)\right\}$-tracially stable.
\end{dfn}

\begin{rmk}\label{finitary}
Self-tracial stability can be expressed in a finitary fashion without the use of ultraproducts.  We believe it is of value to include this alternative framing here, although it seems that the ultraproduct form of the property is much less cumbersome in practice.  The equivalent finitary definition of self-tracial stability can be stated as follows.  Let $(N,\tau)$ be a tracial von Neumann algebra with generating set $\left\{x_1,x_2,\dots\right\}$.  For any $m \in \mb{N}$, let $X_m$ denote the set $\left\{x_1,\dots, x_m\right\}$.  We have that $(N,\tau)$ is self-tracially stable if and only if for any $\varepsilon > 0$ and $n \in \mb{N}$, there are $m>n$ and $k \in \mb{N}$ such that if $\pi: N \rightarrow N$ is a function such that \[||\pi(p(x_1,\dots, x_m)) - p(\pi(x_1),\dots, \pi(x_m))||_2 < k^{-1}\] for every noncommutative $*$-monomial $p$ on $m$ variables of degree no greater than $k$ then there is a unital $*$-homomorphism $\tilde{\pi}: N \rightarrow N$ such that $||\pi(x_j) - \tilde{\pi}(x_j)||_2 < \varepsilon$ for every $1 \leq j \leq m$.  Compare this definition with Definition \ref{tube}.  Note that this definition can be adapted to provide a definition of stability for infinitely generated groups (cf. \S \ref{groups}).
\end{rmk}

%\begin{rmk}
%In \cite{hadshu}, Hadwin-Shulman introduce the concept of tracial stability, and using their terminology, we see that a II$_1$-factor $N$ is self-stable if and only if $N$ is $\left\{N\right\}$-tracially stable.  While Hadwin-Shulman's terminology is more precise and descriptive for general situations and tracial considerations are intrinsic to the theory of II$_1$-factors, we opt for the less cumbersome term \tql self-stable.\tqr
%\end{rmk}

\begin{prop}\label{Rss}
The separable hyperfinite II$_1$-factor $R$ is self-tracially stable.
\end{prop}

\begin{proof}
Let $\pi: R \hookrightarrow R^\mc{U}$ be given. It is well-known that any two embeddings $\rho$ and $\sigma$ of $R$ into $R^\mc{U}$ are unitarily conjugate.  Indeed, by Proposition \ref{waue}, $\rho$ and $\sigma$ are weakly approximately unitarily equivalent.  Then by Theorem 3.1 of \cite{autoultra}, $\rho$ and $\sigma$ are unitarily conjugate.  Let $\rho: R \rightarrow R^\mc{U}$ denote the constant-sequence embedding.  That is, $\rho(a) = (a)_\mc{U}$ for every $a \in R$.  Then for each $a \in R$, $\pi(a) = u \rho(a) u^*$ for some unitary $u \in R^\mc{U}$.  It is well-known that for any unitary $u \in R^\mc{U}$ there exist unitaries $u_k \in R$ such that $u = (u_k)_\mc{U}$.  So we have for each $a \in R$,

\begin{align*}
\pi(a) &= u\rho(a)u^*\\
&= (u_k)_\mc{U} (a)_\mc{U} (u^*_k)_\mc{U}\\
&= (u_k a u_k^*)_\mc{U}.
\end{align*}
Letting $\pi_k(a) = u_k \pi(a) u_k^*$ completes the proof.
\end{proof}

The following theorem shows that modulo CEP, self-tracial stability in fact characterizes amenability for II$_1$-factors. This answers Question 2.8 in \cite{sellin} (modulo CEP, of course). 
%The proof given below relies on Proposition \ref{jungslem}.

\begin{thm}\label{tscon}
Let $(N,\tau)$ be a separably acting tracial von Neumann algebra satisfying CEP.  If $N$ is self-tracially stable then $N$ is semidiscrete.
\end{thm}

\begin{proof}
If $N$ is finite-dimensional, then the conclusion clearly holds.  So assume that $N$ is infinite-dimensional. Let $\tilde{\pi}: N \rightarrow R^\mc{U}$ be an embedding.  Fix an embedding $\iota: R \rightarrow N$ in order to induce an embedding $\pi:= \iota^\mc{U} \circ \tilde{\pi}: N \rightarrow N^\mc{U}$. Since no confusion can occur, we suppress the $\iota$. Fix $\varepsilon >0$ and a finite subset $F = \left\{x_1,\dots, x_n\right\} \subset N$. Since $\pi(N)\subset R^{\mc{U}}\subset N^{\mc{U}}$, by standard approximation arguments, for each $k \in \mb{N}$ there exist unital matrix subalgebras $\mb{M}_{n(k)}\subset R$ such that $\pi(x_j)_k \in \mb{M}_{n(k)}$ for $1 \leq j \leq n$.  

Since $N$ is self-tracially stable there are $*$-endomorphisms $\pi_k: N \rightarrow N$ such that $\pi(x) = (\pi_k(x))_\mc{U}$ for every $x \in N$.  It follows that there is a $K \in \mc{U}$ such that for $k\in K$, \[||\pi(x_j)_k - \pi_k(x_j)||_2 < \varepsilon\] for every $1 \leq j \leq n$.

We now apply Proposition \ref{jungslem} with $N_\varepsilon = N, \pi_\varepsilon = \pi_K,$ and $A_\varepsilon = \mb{M}_{n(K)}$ to conclude that $N$ is semidiscrete.
\end{proof}

In fact, self-tracial stability is equivalent to another formally weaker condition, defined as follows.

\begin{dfn}
Given a class $\ms{C}$ of $C^*$-algebras, a separable tracial von Neumann algebra $(N,\tau)$ is \emph{$\ms{C}$-ucp stable} if for any unital $*$-homomorphism $\ds \pi: N \rightarrow \prod_{k\rightarrow \mc{U}}(\mc{A}_k,\tau_k)$ with $\mc{A}_k \in \ms{C}$ and $\tau_k \in T(\mc{A}_k)$ there exists a sequence of ucp maps $\varphi_k: N \rightarrow \mc{A}_k$ such that $\pi(x) = (\varphi_k(x))_\mc{U}$ for every $x \in N$.  In the case $\ms{C} = \left\{(N,\tau)\right\}$, we say that $N$ is \emph{self-ucp stable}.
\end{dfn}

\begin{thm}\label{tschar}
Let $(N,\tau)$ be a separably acting tracial von Neumann algebra satisfying CEP.  The following are equivalent.

\begin{enumerate}
\item $N$ is injective;
\item $N$ is self-tracially stable;
\item $N$ is self-ucp stable.
\end{enumerate}
\end{thm}

The only non-trivial implication is $(3)\Rightarrow (1)$. This statement was discussed in a different context by Brown in Theorem 6.3.3 of \cite{invar}; we document the proof via injectivity here. 

\begin{proof}[Proof of Theorem \ref{tschar}]
 The implication clearly holds if $N$ is finite-dimensional, so assume $N$ is infinite dimensional. Fix a unital inclusion $\iota_0: R \hookrightarrow N$ to induce an embedding $\iota: R^{\mc{U}} \hookrightarrow N^{\mc{U}}$. Let $\sigma:N\hookrightarrow R^{\mc{U}}$ be an embedding (given by hypothesis). Let $\pi=\iota\circ \sigma: N\rightarrow N^{\mc{U}}$.  Since $N$ is self-ucp stable, there exists a sequence of ucp maps $\varphi_k: N \rightarrow N$ such that $\pi(x) = (\varphi_k(x))_\mc{U}$ for every $x \in N$.  Let $\mb{E}_R$ denote the conditional expectation of $N$ onto $\iota_0(R)$.  Then by construction we have  that $\pi(x) = (\mb{E}_R(\varphi_k(x)))_\mc{U}$.  

We now show that $N$ is injective, and by \cite{connes} this will imply $(1)$.  Let $A \subset B$ be an inclusion of operator systems, and let $\psi: A \rightarrow \sigma(N)$ be a ucp map.  Consider the ucp map $\Phi \circ \psi: A \rightarrow \ell^\infty(R)$ where $\Phi: N \rightarrow \ell^\infty(R)$ is given by $\Phi(x) = (\mb{E}_R \circ \varphi_k(x))_{k=1}^\infty$ for $x \in N$.  Since $\ell^\infty(R)$ is injective, there exists $\tilde{\Phi}: B \rightarrow \ell^\infty(R)$ extending $\Phi \circ \psi$.  Let $\mb{E}_N$ denote the conditional expectation of $R^\mc{U}$ onto $\sigma(N)$, and let $Q: \ell^\infty(R) \rightarrow R^\mc{U}$ denote the canonical quotient map. Then we have that $\tilde{\psi} := \mb{E}_N \circ Q \circ \tilde{\Phi}: B \rightarrow \sigma(N)$ is a ucp extension of $\psi$.  Thus $\sigma(N) \cong N$ is injective.
\end{proof}

\begin{cor}
If $(N,\tau)$ is a separably acting non-amenable tracial von Neumann algebra satisfying CEP, then given an embedding $\pi: N \rightarrow N^\mc{U}$ that factors through $R^\mc{U}$, $\pi$ does not lift.  Furthermore, $\pi$ is not unitarily conjugate to the constant-sequence embedding of $N$ into $N^\mc{U}$.
\end{cor}

\begin{proof}
This follows directly from the proof of Theorem \ref{tschar}.
\end{proof}

%\begin{?}
%Are there embeddings $N \hookrightarrow N^\mc{U}$ that neither factor through $R^\mc{U}$ nor are liftable?
%\end{?}

\begin{rmk}
There are many variations of the above equivalent statements that also characterize amenability modulo CEP.  For instance:
\begin{itemize}
\item For any collection $\ms{C}$ of II$_1$-factors into which $N$ embeds, $N$ is $\ms{C}$-tracially stable;
\item For any II$_1$-factor $M$ into which $N$ embeds, $N$ is $\left\{M\right\}$-tracially stable;
\item There exists a II$_1$-factor $M$ into which $N$ embeds such that $N$ is $\left\{M\right\}$-tracially stable;
\item $N$ is $\left\{R\right\}$-ucp stable;
\item There exists a II$_1$-factor $M$ for which $N$ is $\left\{M\right\}$-ucp stable.
\item $N$ is \textbf{II$_1$}-ucp stable where \textbf{II$_1$} denotes the class of II$_1$-factors.
\end{itemize}
\end{rmk}

One of the appeals of self-tracial stability is that it does not require a CEP assumption to define. We see that the CEP assumption in Theorem \ref{tschar} yields the following corollary, yielding a strategy for identifying counterexamples to CEP.

\begin{cor}
If $(N,\tau)$ is a separable tracial von Neumann algebra that is self-tracially stable and not hyperfinite, then it does not satisfy CEP.
\end{cor}

Considerations of new properties of tracial von Neumann algebras in the context of ultraproducts often leads one to consider model theoretic questions.  A natural question to ask is if self-tracial stability is an axiomatizable property.

\begin{dfn}[\cite{fms}]
A property is \emph{axiomatizable} if it is closed under isomorphism, ultrapower, and ultraroot (cf. \cite{FHS}).
\end{dfn}

\begin{prop}\label{axiom}
Self-tracial stability is not an axiomatizable property.  
\end{prop}

\begin{proof}
By way of contradiction, suppose that self-tracial stability is axiomatizable. Let $(N,\tau)$ be a non-hyperfinite separably acting tracial von Neumann algebra such that $N$ is elementarily equivalent to $R$.  This implies that $N$ is self-tracially stable because $R$ is.  By elementary equivalence, $N^\mc{U} \cong R^\mc{U}$, and thus $N$ satisfies CEP.  So by Theorem \ref{tschar} we have that $N$ is not self-tracially stable, absurd!
\end{proof}

We close this section by mentioning that since self-tracial/ucp stability is a property that does not require any sort of CEP assumption to state, results regarding self-tracial/ucp stability away from the context of CEP would be of significant interest.  A starting point would be to consider tracial von Neumann algebras with Property (T).

\section{Conjugation by ucp maps}\label{ucp}

%\subsection{Complete tubularity and ucp-conjugacy} In this section we introduce and investigate the notion of \emph{ucp-conjugacy}. 

\begin{dfn}\label{wms}
Let $(N,\tau)$ be a separably acting tracial von Neumann algebra, and for each $k \in \mb{N}$ let $M_k$ be a II$_1$-factor. We say that two embeddings $\ds\pi,\rho:N\rightarrow \prod_{k\rightarrow \mc{U}}M_k$ are \emph{ucp-conjugate} if there exists a sequence of subtracial ucp maps $\varphi_k: M_k \rightarrow M_k$ such that for any $x\in N$ we have $\rho(x)=(\varphi_k(\pi(x)_k))_{\mc{U}}$. We write $\rho=(\varphi_k)_\mc{U} \circ \pi$. 
\end{dfn}

\noin Unless mentioned otherwise, all ucp maps in the rest of the paper are assumed to be subtracial.

%Observe that one has the following immediate corollaries. 

\begin{prop}
If $(N,\tau)$ is self-tracially stable, then any pair of embeddings of $N$ into $N^\mc{U}$ are ucp-conjugate. 
\end{prop}
\begin{proof}
For any two embeddings $\sigma_1,\sigma_2:N\rightarrow N^{\mc{U}}$, one has that for each $k\in \mb{N}$, there exist unital $*$-homomorphisms $\pi_{1k}$, $\pi_{2k}:N\rightarrow N$ satisfying $\sigma_i(x)=(\pi_{ik}(x))_{\mc{U}}$. Define $\varphi_k:N\rightarrow N$ as $\varphi_k= \pi_{2k}\circ\pi_{1k}^{-1}\circ\mb{E}_k$ where $\mb{E}_k$ is the conditional expectation from $N$ onto $\pi_{1k}(N)$. It is easy to see that this collection $(\varphi_k)_{k \in \mb{N}}$ of ucp maps implement $\sigma_2=(\varphi_k)_\mc{U} \circ \sigma_1$, as required.
\end{proof}

From the functional calculus fact that unitaries in ultraproducts lift to a sequence of unitaries, it follows that if $N$ satisfies the property that any two embeddings of $N$ into $R^{\mc{U}}$ are conjugate by a unitary in $R^{\mc{U}}$, then any two embeddings of $N$ into $R^\mc{U}$ are ucp-conjugate. In \cite{jung}, Jung showed that assuming CEP and finitely many generators, $R$ is the only separably acting II$_1$-factor with the former property; the goal of this section is to show that, modulo CEP, $R$ is the only separably acting II$_1$-factor with the latter property.  This considerably strengthens Jung's theorem.
%\begin{proof}
%This is straightforward. For $\sigma_1,\sigma_2: N\to R^{\mc{U}}$, we have from Jung's result, there exists a unitary $u\in R^{\mc{U}}$ so that $\sigma_1=u\sigma_2u^*$. Let ${u_k}_{\mc{U}}$ be a unitary lift of $u$, so that $u_k\in R$. We clearly have for the maps $\phi_k:R\to R$, given by $\phi_k(x)=u_kxu_k^*$, the condition $\sigma_2=\phi\circ \sigma_1$.
%\end{proof}
%These corollaries lead us to the question of whether the weaker property of weak-stability is sufficient to characterize hyperfiniteness. We show that this is indeed true if we assume also that $N$ is Connes embeddable. The proof of this result is along the same lines as Jung's. 
To do this we define a ucp analog of Jung's notion of tubularity that allows for the possibility of infinitely many generators.

\begin{dfn}\label{tube}
Let $(N,\tau)$ be a tracial von Neumann algebra, and let $\{x_1,\hdots, x_n\}$ be a finite subset of $(N)_{\leq 1}$. Fix $\varepsilon > 0$. The set $\left\{x_1,\dots, x_n\right\}$ is \emph{$\varepsilon$-completely tubular} if there exist $k\in \mb{N}$ and $\gamma>0$ with the property that for any $J\in \mb{N}$ and $\xi,\eta \in \Gamma(\left\{x_1,\dots, x_n\right\};k,J,\gamma)$, there exists a ucp map $\varphi: \mb{M}_J \rightarrow \mb{M}_J$ such that \[||\xi_j-\varphi(\eta_j)||_{2,\mb{M}_j}<\varepsilon\] for every $1 \leq j \leq n$. If $\left\{x_1,\dots, x_n\right\}$ is $\varepsilon$-completely tubular for every $\varepsilon >0$, we say that $\left\{x_1,\dots,x_n\right\}$ is \emph{completely tubular}.

For infinitely many generators, let $X = \left\{x_1, x_2, \dots\right\} \subset (N)_{\leq 1}^\mb{N}$ be a sequence of generators of $N$.  For $m \in \mb{N}$, let $X_m := \left\{x_1,\dots,x_m\right\}$.  We say that $X$ is \emph{completely tubular} if for any $\varepsilon > 0$ and $n \in \mb{N}$, there is an $m > n$ such that $X_m$ is $\varepsilon$-completely tubular as defined above. 
\end{dfn}

\begin{rmk}
For a separable tracial von Neumann algebra, one can always take a countable $||\cdot||_2$-dense subset of the unit ball to serve as the generators.
\end{rmk}

\begin{lem}\label{tube2}
Let $(N,\tau)$ be a separably acting tracial von Neumann algebra satisfying CEP. If any pair of embeddings of $N$ into $R^\mc{U}$ are ucp-conjugate, then for any countable $||\cdot||_2$-dense subset $X= \left\{x_j\right\}_{j \in \mb{N}} \subset (N)_{\leq 1}, X$ is completely tubular.
\end{lem}
\begin{proof}
Suppose not. Then for some dense subset $X := \left\{x_j\right\}_{j \in \mb{N}} \subset (N)_{\leq 1}$, there is an $\varepsilon > 0$ and an $n_0 \in \mb{N}$ such that for every $k > n_0, X_{k}:=\left\{x_1,\dots, x_{k}\right\}$ is not $\varepsilon$-completely tubular. So for $k> n_0$ there exist  $J_k\in \mb{N}$ and a pathological pair of microstates $\xi^{(k)},\eta^{(k)}\in \Gamma(X_k;k,J_k,k^{-1})$ so that for any choice of ucp map $\varphi: \mb{M}_{J_k}  \rightarrow  \mb{M}_{J_k} $, \[||\xi^{(k)}_j-\varphi(\eta^{(k)}_j)||_{2,\mb{M}_{J_k}}\geq\varepsilon\] for every $1 \leq j \leq k$. 
%We want to contradict this by finding a large enough $p\in \mb{N}$ and $m\in \mb{N}$ and show that $\xi_m$ is near the image of a $D$-bounded linear map of $\eta_m$. For this we need two embeddings to apply our hypothesis. 
Consider $\mb{M}_{J_k} \subset R$ as a unital subalgebra. Our choice of the generalized microstates $\xi^{(k)}$ and $\eta^{(k)}$ gives access to two embeddings $\sigma_1, \sigma_2: N \rightarrow R^{\mc{U}}$ given by $\sigma_1(x_j)= (\xi^{(k)}_j)_{\mc{U}}$ and  $\sigma_2(x_j)= (\eta^{(k)}_j)_{\mc{U}}$ for every $j \in \mb{N}$ (where $\xi^{(k)}_j = \eta^{(k)}_j = 0$ if $j>k$). From the hypothesis, there exist ucp maps $\varphi_k:R\to R$ such that $(\varphi_k(\xi^{(k)}_j))_{\mc{U}}=(\eta^{(k)}_j)_{\mc{U}}$ for every $j \in \mb{N}$. Thus there is a $K \in \mc{U}$ so that for all $k \in K$ \[||\varphi_k(\xi^{(k)}_j) - \eta^{(k)}_j||_{2,R}<\frac{\varepsilon}{2}\] for every $1 \leq j \leq k$. Let $\mb{E}_{\mb{M}_{J_k}}$ denote the conditional expectation of $R$ onto $\mb{M}_{J_k}$.  It follows that \[||\mb{E}_{\mb{M}_{J_k}}(\varphi_k(\xi^{(k)}_j)) - \eta^{(k)}_j||_2 < \frac{\varepsilon}{2}\] for every $1 \leq j \leq k$.   Then for $k \in K \cap \left\{k: k > n_0\right\} (\neq \emptyset)$, we have that the ucp map $\psi_k:=\mb{E}_{\mb{M}_{J_k}} \circ \varphi_k|_{\mb{M}_{J_k}}: \mb{M}_{J_k} \rightarrow \mb{M}_{J_k}$ satisfies \[||\psi_k(\xi^{(k)}_j) - \eta^{(k)}_j||_2 < \varepsilon\] for every $1 \leq j \leq k$--a contradiction.
\end{proof}

\begin{lem}\label{tube3}
Let $(N,\tau)$ be a separably acting tracial von Neumann algebra satisfying CEP. If for any countable $||\cdot||_2$-dense subset $X = \left\{x_j\right\}_{j \in \mb{N}} \subset (N)_{\leq 1}, X$ is completely tubular, then $N$ is injective.
\end{lem}
\begin{proof}
We wish to apply Proposition \ref{kishimoto}.  So fix a finite subset $F \subset (N)_{\leq 1}$ and $\varepsilon > 0$.  Let $X:= \left\{x_j\right\}_{j \in \mb{N}}$ be a countable $||\cdot||_2$-dense subset of $(N)_{\leq 1}$ such that $F \subset X$.  For each $k \in \mb{N}$, define $X_k:=\left\{x_1,\dots, x_k\right\}$.

Since $N$ satisfies CEP, for each $k \in \mb{N}$ there exists \[\xi^{(k)} = (\xi^{(k)}_1,\dots,\xi^{(k)}_k) \in \Gamma(X_k;k,J_k,k^{-1}).\]  Write $R=R_1\otimes R_2$ for $R_i \cong R, i = 1,2$, and for $j \in \mb{N}$ define $y_j:=(1_{R_1}\otimes \xi^{(k)}_j)_{\mc{U}}\in R^{\mc{U}}$. Construct an embedding $\pi: N\to R^{\mc{U}}$ given by $\pi(x_j)=y_j$ for $j \in \mb{N}$. Let $k_0$ be such that $F \subset X_{k_0}$ and $X_{k_0}$ is $\frac{\varepsilon}{2}$-completely tubular.  So there exists $k_1 \in \mb{N}$ such that for any $J \in \mb{N}$ and any $\xi, \eta \in \Gamma(X_{k_0}; k_1, J, k_1^{-1})$ there is a ucp map $\varphi: \mb{M}_J \rightarrow \mb{M}_J$ such that \[||\varphi(\eta_j) - \xi_j||_2 < \frac{\varepsilon}{2}\] for $1 \leq j \leq k_0$. Fix $J \in \mb{N}$ so that $\Gamma(X_{k_0}; k_1, J, k_1^{-1})$ is nonempty and let $\eta = (\eta_1,\dots, \eta_{k_0}) \in \Gamma(X_{k_0}; k_1, J, k_1^{-1})$. For each $k \geq \max\left\{k_0,k_1\right\}$, note that $(\eta_1\otimes I_{J_k},\dots, \eta_{k_0}\otimes I_{J_k})$ and $(I_{J} \otimes \xi^{(k)}_1,\dots, I_J\otimes \xi^{(k)}_{k_0})$ are both microstates in $\Gamma(X_{k_0};k_1, J \cdot J_k, k_1^{-1})$.   So by $\frac{\varepsilon}{2}$-complete-tubularity, there is a ucp map $\varphi_k: \mb{M}_J \otimes \mb{M}_{J_k}  \rightarrow \mb{M}_J \otimes \mb{M}_{J_k}$  such that \[|| \varphi_k(\eta_j\otimes I_{J_k}) - I_J\otimes \xi^{(k)}_j ||_{2, \mb{M}_{J\cdot J_k}}< \frac{\varepsilon}{2}\] for $1 \leq j \leq k_0$.  We define the ucp map $\rho: \mb{M}_J\rightarrow R^{\mc{U}} $ given by $\rho(x)= (\varphi_k(x\otimes I_{J_k}))_{\mc{U}}$ (here we consider $\mb{M}_{J} \subset R_1$ and $\mb{M}_{J_k} \subset R_2$).  Finally note that $\rho$ satisfies $\inf_{B\in M_J}||y_j- \rho(B)||_{2, R^\mc{U}}\leq \frac{\varepsilon}{2} < \varepsilon$ for every $1 \leq j \leq k_0$.  Thus by composing $\rho$ with the conditional expection onto $N$, we can apply Proposition \ref{kishimoto}, and it follows that $N$ is injective.
%More specifically, we have $$||\rho(\zeta_j)-y_j||_2 = \lim_{\mc{U}} || I\otimes \xi_{jm}\otimes I -T_m(\zeta_j\otimes I\otimes I)||_2\leq \varepsilon$$ Therefore, for any $\varepsilon$, every element of $\pi(X)$ is encapsulated by the $\varepsilon$ neighborhood in the 2-norm of the finite dimensional subalgebra $\rho(A)$ of the tracial von Neumann algebra $R^{\mc{U}}$. Therefore, $N$ is semidiscrete. 
\end{proof}

Lemmas \ref{tube2} and \ref{tube3} combine to prove the following theorem.
 
\begin{thm}\label{jungv2}
Let $(N,\tau)$ be a separably acting tracial von Neumann algebra satisfying CEP. Then any pair of embeddings of $N$ into $R^\mc{U}$ are ucp-conjugate if and only if $N$ is injective.  
%Moreover, if $N$ is finitely generated with finite generating set $X \subset (N)_{\leq 1}$, then $N$ is injective if and only if $X$ is completely tubular.
\end{thm}
%\begin{proof}
%From lemmas \ref{tube1}, \ref{tube2} and \ref{tube3} and the fact that amenability implies tubularity implies complete tubularity, we have the result.
%\end{proof}

By way of conditional expectations, we obtain the following more general version of Theorem \ref{jungv2}.

\begin{cor}\label{jungv2cor}
Let $(N,\tau)$ be a separably acting tracial von Neumann algebra satisfying CEP, and for each $k \in \mb{N}$ let $M_k$ be a II$_1$-factor. Then any two embeddings of $N$ into $\ds\prod_{k\rightarrow \mc{U}}M_k$ are ucp-conjugate if and only if $N$ is injective.  In particular, any two embeddings of $N$ into $\ds\prod_{k\rightarrow \mc{U}}M_k$ are unitarily conjugate if and only $N$ is injective.
\end{cor}
\begin{proof}
Assume that any two embeddings of $N$ into $\ds\prod_{k\rightarrow \mc{U}}M_k$ are ucp-conjugate. Let $\sigma_1,\sigma_2: N\rightarrow R^{\mc{U}}$ be two embeddings. Extend these embeddings to $\ds\tilde{\sigma}_i: N\rightarrow \prod_{k\rightarrow \mc{U}}M_k$ through inclusions $\iota_k: R\rightarrow M_k$.  That is, $\tilde{\sigma}_i = (\iota_k)_\mc{U} \circ \sigma_i, i = 1,2$. From the assumption there are ucp maps $\varphi_k: M_k\to M_k$, satisfying $(\varphi_k)_\mc{U}\circ \tilde{\sigma}_1=\tilde{\sigma}_2$. Let $\mb{E}_{\iota_k(R)}$ denote the conditional expectation from $M_k$ onto $\iota_k(R)$. Then taking $\psi_k:=\mb{E}_{\iota_k(R)}\circ\varphi_k|_{\iota(R)}$ we obtain $(\psi_k)_\mc{U} \circ \sigma_1 = \sigma_2$. From Theorem \ref{jungv2}, $N$ is injective. 

%From a standard functional calculus argument we have that any unitary $u$ of $\prod_{k\rightarrow \mc{U}}M_k$ lifts into $(u_i)_{\mc{U}}$, where $u_i\in M_i$ are unitaries. But note Ad$u_i$ is a unital completely positive map on $M_i$ so the previous argument applies. 

If $N$ is injective, it is hyperfinite by \cite{connes}, and any two embeddings of $N$ are unitarily conjugate.

The \tql In particular\tqr part of the statement of this Corollary follows from the fact that any unitary $\ds u \in \prod_{k\rightarrow \mc{U}}M_k$ is of the form $u = (u_k)_\mc{U}$ where $u_k$ is a unitary in $M_k$.
\end{proof}

In \cite{saa}, one can find a separable formulation of Jung's theorem, stated as follows.

\begin{thm}[\cite{saa}]\label{saasep}
Let $(N,\tau)$ be a separably acting tracial finitely generated von Neumann algebra satisfying CEP.  Then $N$ is amenable if and only if for every separable McDuff II$_1$-factor $M$, any two embeddings $\pi,\rho: N \rightarrow M$ are weakly approximately unitarily equivalent.
\end{thm}

\noin  Theorem \ref{jungv2} is a ucp version of Jung's theorem, and Theorem \ref{saasep} is a separable version of Jung's theorem.  Thus it is natural to consider a separable version of Theorem \ref{jungv2} (or a ucp version of Theorem \ref{saasep}).  We therefore ask the following question.

\begin{?}\label{waucpq}
Let $(N,\tau)$ be a separably acting tracial von Neumann algebra satisfying CEP.  If for any II$_1$-factor $M$ and any two embeddings $\pi, \rho: N \rightarrow M$ there is a sequence of ucp maps $\varphi_n: M \rightarrow M$ such that for every $x \in N,$ $\ds \lim_{n\rightarrow \infty} ||\pi(x) - \varphi_n(\rho(x))||_2 =0$, then does it follow that $N$ is amenable?
\end{?}

\noin The proof of Theorem \ref{saasep} uses the fact that given a unitary $u \in R^\mc{U}$, there is a sequence of unitaries $u_k \in R$ such that $u = (u_k)_\mc{U}$.  Thus the map $\text{Ad} u$ given by $\text{Ad} u (x) = u^*x u$ can be expressed as $\text{Ad} u = (\text{Ad} u_k)_\mc{U}$, and Jung's theorem can be converted into separable language accordingly.  Unfortunately, this approach breaks down in the more general ucp setting, because given a ucp map $\varphi: R^\mc{U} \rightarrow R^\mc{U}$, it does not necessarily decompose as $\varphi = (\varphi_k)_\mc{U}$ for ucp maps $\varphi_k: R \rightarrow R$.  Indeed, any conditional expectation of $R^\mc{U}$ onto a nontrivial separable subfactor will not have such a decomposition.  With this obstruction, a resolution of the above question using existing techniques is not clear. 

Based on our discussion of  ucp-conjugation, it is natural to consider the relationship between ucp-conjugation and unitary equivalence.  To ensure we have an equivalence relation in terms of ucp-conjugation, we make the following definition.

\begin{dfn}
Let $(N,\tau)$ be a tracial von Neumann algebra, and for each $k \in \mb{N}$ let $M_k$ be a II$_1$-factor.  Two embeddings $\ds \pi, \rho: N \rightarrow \prod_{k\rightarrow \mc{U}}M_k$ are \emph{doubly ucp-conjugate} or \emph{ucp-equivalent} if there are ucp maps $\varphi_k, \psi_k: M_k \rightarrow M_k$ such that $\pi = (\varphi_k)_\mc{U}\circ \rho$ and $\rho = (\psi_k)_\mc{U} \circ \pi$.
\end{dfn}

It is also of worth to ask about the relationship between ucp-conjugation and unitary conjugation in our setting. The following example shows that the notions of unitary equivalence and ucp-equivalence do not coincide in general.

\begin{exmpl}
Consider the two embeddings $\pi,\rho: L(\mb{F}_2) \rightarrow L(\mb{F}_2)^\mc{U}$ where $\pi$ is the constant-sequence embedding, and $\rho = \alpha^\mc{U} \circ \pi$ where $\alpha \in \text{Aut}(L(\mb{F}_2))$ is the automorphism on $L(\mb{F}_2)$ induced by involuting the two generators.  Note that $\alpha$ is order two, so $\pi = \alpha^\mc{U} \circ \rho$. It is well known that $\alpha$ is not approximately inner in the point-$||\cdot||_2$ topology.  It follows that $\pi$ and $\rho$ are not unitarily equivalent, but since $\alpha$ is ucp, these two embeddings are ucp-equivalent.
\end{exmpl}

\noin With the above example in mind, we formulate the following question.

\begin{?}
If two embeddings $\pi, \rho:N\to R^{\mc{U}}$ are ucp-equivalent, are they unitarily equivalent? 
\end{?}

\noin An answer to this question in the affirmative in combination with Jung's theorem would yield Theorem \ref{jungv2} as an immediate corollary.

It is of interest to ask the question in terms of general automorphisms of the ultraproduct.

\begin{?}\label{autoq}
Let $(N,\tau)$ be a separably acting tracial von Neumann algebra satisfying CEP. If any two embeddings of $N$ (separable) into $\ds \prod_{k\rightarrow \mc{U}}M_k$ are conjugate by an automorphism of $\ds \prod_{k\rightarrow \mc{U}}M_k$, then is $N$ amenable? 
\end{?}

\noin Note that Corollary \ref{jungv2cor} answers Question \ref{autoq} for inner automorphisms and automorphisms that lift to ucp maps. 
%The general case is a very interesting and hard open problem. 
In \cite{genjung} the authors together with I. Goldbring make significant progress on this question.

%Perhaps we should consider adapting the finite tubularity proofs that Jung does, in our setup? 
%
%A similar statement holds for ultraproducts.
%
%
%
%There is a similar statement in the group context.

\section{Popa's cardinality question}

Recall that $\HOM(N,M)$ denotes the space of embeddings of $N$ into $M$ modulo unitary equivalence, and $[\pi]$ denotes the unitary equivalence class associated to $\pi: N \rightarrow M$. Given $[\pi_n],[\pi] \in \HOM(N, M)$, we say $[\pi_n] \rightarrow [\pi]$ if there exist $\pi_n' \in [\pi_n]$ such that for any $x \in N$, $||\pi_n'(x) - \pi(x)||_2 \rightarrow 0$.  The following metric induces this topology.  Fix a countable generating (e.g., dense) subset $\left\{x_j\right\}_{j=1}^\infty \subset N_{\leq 1}$; given $[\pi], [\rho] \in \HOM(N,M)$ put \[d_M([\pi], [\rho]) = \inf_{u \in \mc{U}(M)} \left(\sum_{j=1}^\infty 2^{-2j} \left|\left|\pi(x_j) - u\rho(x_j)u^*\right|\right|_2^2\right)^\frac{1}{2}.\]

In \cite{popa}, Popa asks the following question.

\begin{b?}[\cite{popa}]
If $N$ is a separable von Neumann subalgebra of an ultraproduct II$_1$-factor $\ds \prod_{k\rightarrow \mc{U}}M_k$, then how large is $\ds\HOM(N,\prod_{k\rightarrow \mc{U}}M_k)$?
\end{b?}

\noin As an immediate response to this question, we have the following consequence of Corollary \ref{jungv2cor}.

\begin{cor}Modulo CEP, if $N$ is a non-amenable separable finite von Neumann algebra, then $\ds |\HOM(N, \prod_{k\rightarrow \mc{U}}M_k)|\geq 2$. 
\end{cor}

In \cite{topdyn}, Ozawa shows in Theorem A.1 that modulo CEP, if $N$ is not amenable, then $\HOM(N, R^\mc{U})$ is not separable. In Theorem \ref{mcd} below, we use techniques from \cite{saa} to extend Ozawa's result, answering Popa's question.  Some preparation is in order before presenting the theorem.

%It will be useful to keep the following finitary formulation of weak approximate unitary equivalence in mind: for $\pi,\rho: N\rightarrow M$, $\pi \sim_\text{waue} \rho$ if and only if for every finite subset $F \subset N_{\leq 1}$ and every $\varepsilon > 0$ there is a unitary $u \in \mathcal{U}(M)$ such that \[||\pi(x) - u\rho(x)u^*||_2 < \varepsilon\] for every $x \in F$.
%
%
%The following theorem shows that weak approximate unitary equivalence of representations of separable algebras in ultraproducts is the same as unitary equivalence.  It was initially proved in the ultrapower setting, but one can see immediately how to extend the result to ultraproducts.
%
%\begin{thm}[\cite{autoultra}]\label{waue=ue}
%Let $N$ be a separable von Neumann algebra and for each $k \in \mb{N}$, let $M_k$ be a separable II$_1$-factor.  If $\pi, \rho: N \rightarrow \prod_{k\rightarrow \mc{U}}M_k$ are weakly approximately unitarily equivalent, then they are unitarily equivalent.
%\end{thm}

\begin{dfn}[\cite{newproof}]
Let $N$ be a II$_1$-factor.  For $n \in \mathbb{N}$ and $\delta > 0$, two $n$-tuples $(u_1,\dots,u_n)$ and $(v_1,\dots,v_n)$ of unitaries in $N$ are \emph{$\delta$-related} if there is a sequence $\left\{a_j\right\} \subset N$ with
\begin{align}\label{summing}
\sum_j a_j^*a_j &=1_N = \sum_j a_ja_j^*
\end{align}
(where the convergence is SOT) such that for every $1 \leq k \leq n$,
\begin{align}\label{ineq}
\sum_j ||a_j u_k - v_ka_j||_2^2 &< \delta.
\end{align}
  We say that $\left\{a_j\right\}$ is a sequence that witnesses that $(u_1,\dots,u_n)$ and $(v_1,\dots,v_n)$ are $\delta$-related.
\end{dfn}

The next theorem was proved by Haagerup in \cite{newproof}.  The statement we present is slightly different from Haagerup's original wording.  The present version indicates the uniformity with which the $\varepsilon$ and $\delta(n, \varepsilon)$ estimates can be made and expands the scope to all finite factors.  Haagerup's proof of Theorem \ref{deltarelated} can be easily adapted (essentially exchanging ultrapowers for ultraproducts and taking matrix algebras into consideration) to prove this moderately stronger version.

\begin{thm}[\cite{newproof}]\label{deltarelated}
Fix $n \in \mb{N}$.  For every $\varepsilon > 0$ there exists a $\delta(n,\varepsilon) > 0$ such that for any finite factor $N$ and any two $\delta(n,\varepsilon)$-related $n$-tuples of unitaries $(u_1,\dots,u_n)$ and $(v_1,\dots,v_n)$ in $N$, there exists a unitary $w \in N$ such that for every $1 \leq j \leq n$ \[||wu_j - v_jw||_2 < \varepsilon.\]
\end{thm}

\noin We will also need the following fact from \cite{saa}.  Again, we expand the scope to all separable finite factors; the proof in \cite{saa} still applies to this more general setting.

\begin{lem}[\cite{saa}]\label{tensordelta}
Let $N_1$ and $N_2$ be separable finite factors, and let $(u_1,\dots,u_n)$ and $(v_1,\dots, v_n)$ be two $n$-tuples of unitaries in $N_1$.  Fix $\delta >0$, and let $z \in N_1 \otimes N_2$ be a unitary of the form \[z = \sum_{j=1}^\infty a_j \otimes b_j\] where $\left\{b_j\right\} \subset N_2$ is an orthonormal basis in $L^2(N_2)$ and the convergence is with respect to the $||\cdot||_2$-norm.  If $z$ is such that for every $1 \leq k \leq n$, \[||z(u_k\otimes 1_{N_2}) - (v_k\otimes 1_{N_2})z||_2^2 < \delta,\] then $(u_1,\dots,u_n)$ and $(v_1,\dots,v_n)$ are $\delta$-related.  Furthermore, $\left\{a_j\right\}$ is a sequence that witnesses that $(u_1,\dots,u_n)$ and $(v_1,\dots, v_n)$ are $\delta$-related.
\end{lem}

We are now ready to present the answer to Popa's question.  We should note that after circulating an early preprint of this article containing this result for ultraproducts of McDuff II$_1$-factors, A. Ioana pointed out a way to extend our argument to the general case.  We wish to express our deep gratitude for his contribution.  We will exploit the following fact.

\begin{prop}
Let $X$ and $Y$ be metric spaces, and let $\varphi: X \rightarrow Y$ be a continuous function.  Suppose that for every $\varepsilon > 0$ there is a $\delta > 0$ such that if $x,x' \in X$ are such that $d_Y(\varphi(x),\varphi(x')) < \delta$ then $d_X(x,x') < \varepsilon$.  Then $\varphi$ is a homeomorphism onto its image.
\end{prop}

\begin{thm}\label{mcd}
Let $N$ be a non-amenable finite von Neumann algebra satisfying CEP, and for each $k \in \mb{N}$, let $M_k$ be a II$_1$-factor.  Then $\HOM(N,R^\mc{U})$ embeds into $\ds\HOM(N,\prod_{k\rightarrow \mc{U}}M_k)$.  In particular, $\ds\HOM(N,\prod_{k\rightarrow \mc{U}}M_k)$ is non-separable.
\end{thm}

\begin{proof}
For each $k \in \mb{N}$, fix a copy of the separably acting hyperfinite II$_1$-factor $R$ as a subfactor of $M_k$, and use these inclusions to induce an inclusion $\ds\iota: R^\mc{U} \subset \prod_{k\rightarrow \mc{U}}M_k$.  Consider the map $\ds\varphi: \HOM(N,R^\mc{U}) \rightarrow \HOM(N,\prod_{k\rightarrow \mc{U}}M_k)$ given by $\varphi([\pi]) = [\iota \circ \pi]$.

We will apply the above proposition to $\varphi$.  Note that $\varphi$ is continuous. Let $\left\{u_j\right\}_{j=1}^\infty$ be a generating set of unitaries in $N$, and use this sequence to define the metrics $d_{R^\mc{U}}$ and $\ds d_{\prod_{k\rightarrow \mc{U}}M_k}$. Fix $\varepsilon_0 > 0$. Let $J \in  \mb{N}$ be such that $\sum_{j= J+1}^\infty 2^{2-2j} < \frac{\varepsilon_0^2}{2}$.  By Theorem \ref{deltarelated}, there is a $\delta\left(J, \frac{\varepsilon_0}{\sqrt{2J}}\right)$ such that for any finite factor $\tilde{N}$, and any pair of $\delta\left(J, \frac{\varepsilon_0}{\sqrt{2J}}\right)$-related $J$-tuples of unitaries $(\tilde{u}_1,\dots, \tilde{u}_J), (\tilde{v}_1,\dots, \tilde{v}_J) \in \mc{U}(\tilde{N})$, there is a unitary $\tilde{w} \in \mc{U}(\tilde{N})$ such that for every $1 \leq j \leq J$, $\ds||\tilde{w}\tilde{u}_j-\tilde{v}_j\tilde{w}||_2 < \frac{\varepsilon_0}{\sqrt{2J}}.$
Let $\ds \delta_0^2 = 2^{-2J}\cdot\delta\left(J, \frac{\varepsilon_0}{\sqrt{2J}}\right)$, and let $[\pi], [\rho] \in \HOM(N,R^\mc{U})$ be such that $\ds d_{\prod_{k\rightarrow \mc{U}}M_k}([\iota \circ\pi], [\iota\circ\rho])^2 < \delta_0^2.$ For each $k \in \mb{N}$  there is a matrix subalgebra $\mb{M}_{n_k} \subset R$ such that $\pi(u_j)_k, \rho(u_j)_k \in \mb{M}_{n_k}\subset R$ for every $1 \leq j \leq J$, and we can and do assume that $\pi(u_{j})_k$ and $\rho(u_{j})_k$ are unitaries.  Recall that we fixed embeddings $R \subset M_k$ at the beginning of the proof. So we have $\mb{M}_{n_k} \subset R \subset M_k$. For $k\in \mb{N}$, let $B_k = \mb{M}_{n_k}' \cap M_k$.  It is well-known that $B_k$ is a II$_1$-factor and that $M_k \cong \mb{M}_{n_k} \otimes B_k$ since $B_k$ is the relative commutant of a matrix algebra in a II$_1$-factor.   So for every $1 \leq j \leq J$ and every $k \in \mb{N}$ we have that $\iota(\pi(u_j))_k = \pi(u_j)_k \otimes 1_{B_k}$ and $\iota(\rho(u_j))_k = \rho(u_j)_k \otimes 1_{B_k}$ are unitaries in $\mb{M}_{n_k} \otimes 1_{B_k} \subset \mb{M}_{n_k} \otimes B_k = M_k$.

%For $\gamma \in \left\{\pi,\rho\right\}$, let $\ds\gamma': N \rightarrow \prod_{k\rightarrow \mc{U}}\mb{M}_{n_k}$ be identical to $\gamma$ with the only change being that the codomain of $\gamma'$ is the indicated ultraproduct of matrix algebras whereas the codomain of $\gamma$ is all of $\ds \prod_{k\rightarrow \mc{U}}M_k$.  Then for $\gamma \in \left\{\pi,\rho\right\}$, we decompose $\gamma$ as \[\gamma = \gamma' \otimes 1_{\prod_{k\rightarrow \mc{U}}B_k}: N \rightarrow \prod_{k\rightarrow \mc{U}}(\mb{M}_{n_k} \otimes B_k) = \prod_{k\rightarrow \mc{U}}M_k.\]

 Since $\ds d_{\prod_{k\rightarrow \mc{U}}M_k}([\iota\circ \pi], [\iota\circ\rho])^2 < \delta_0^2$, there is a unitary $\ds z \in \mc{U}\left(\prod_{k\rightarrow \mc{U}}M_k\right) = \mc{U}\left(\prod_{k\rightarrow \mc{U}}\mb{M}_{n_k} \otimes B_k\right)$ such that $\ds \sum_{j=1}^\infty ||z\cdot \iota(\pi(u_j)) - \iota(\rho(u_j))\cdot z||_2^2 <\delta_0^2$.
 
%\begin{align*}
%\sum_{j=1}^\infty ||\pi(u_j) - z^*\rho(u_j)z||_2^2 
%%& = \sum_{j=1}^\infty ||z(\pi'\otimes 1_{\prod_{k\rightarrow \mc{U}}B_k})(u_j) - (\rho'\otimes 1_{\prod_{k\rightarrow \mc{U}}B_k})(u_j)z||_2^2 \\
%&<\frac{\delta\left(J, \frac{\varepsilon_0}{\sqrt{2J}}\right)}{2^{2J}}
%\end{align*}

By standard approximation arguments, we may assume that $z$ has the form $z = (z_k)_\mc{U}$ where $z_k \in \mc{U}(\mb{M}_{n_k} \otimes B_k)$ and $\ds z_k = \sum_{m=1}^\infty a_{mk}\otimes b_{mk}$ where $\left\{b_{mk}\right\}$ is an orthonormal basis for $L^2(B_k)$ and convergence is with respect to the $||\cdot||_2$-norm.  There is a $K \in \mc{U}$ such that for $k \in K$ and $1\leq j' \leq J$, we have
\begin{align*}
&2^{-2j'}||z_k(\pi(u_{j'})_k \otimes 1_{B_k}) - (\rho(u_{j'})_k\otimes 1_{B_k})z_k||_2^2\\
&=2^{-2j'}||z_k\iota(\pi(u_{j'}))_k - \iota(\rho(u_{j'}))z_k||_2^2 \\
 & \leq \sum_{j=1}^\infty 2^{-2j}||z_k\iota(\pi(u_j))_k - \iota(\rho(u_j))_kz_k||_2^2\\
& < {2^{-2J}}\cdot\delta\left(J, \frac{\varepsilon_0}{\sqrt{2J}}\right).
\end{align*}

Thus for $k \in K$ and $1 \leq j' \leq J$, \[||z_k(\pi(u_{j'})_k \otimes 1_{B_k}) - (\rho(u_{j'})_k\otimes 1_{B_k})z_k||_2^2 < \delta\left(J, \frac{\varepsilon_0}{\sqrt{2J}}\right).\]

By Lemma \ref{tensordelta}, we have that for each $k \in K$, $(\pi(u_1)_k,\dots, \pi(u_J)_k)$ and \linebreak $(\rho(u_1)_k,\dots, \rho(u_J)_k)$ are $\delta\left(J, \frac{\varepsilon_0}{\sqrt{2J}}\right)$-related.  So by Theorem \ref{deltarelated}, for each $k \in K$, there is a unitary $w_k \in \mc{U}(\mb{M}_{n_k})$ such that \[||w_k\pi(u_{j'})_k - \rho(u_{j'})_kw_k||_2^2 < \frac{\varepsilon_0^2}{2J}\] for every $1 \leq j' \leq J$.  Let $w = (w_k)_\mc{U}$. Then we have
\begin{align*}
d_{R^\mc{U}}([\pi],[\rho])^2 & \leq \sum_{j=1}^\infty 2^{-2j}||w\pi(u_j) - \rho(u_j)w||_2^2\\
& = \sum_{j'=1}^J 2^{-2j'}||w\pi(u_{j'}) - \rho(u_{j'}) w||_2^2 + \sum_{j=J+1}^\infty 2^{-2j}||w\pi(u_j) - \rho(u_j)w||_2^2\\
& < J \cdot \frac{\varepsilon_0^2}{2J} + \frac{\varepsilon_0^2}{2}\\
& = \varepsilon_0^2.\qedhere
\end{align*}
\end{proof}

\section{Commuting embeddings and amenability}

\subsection{Commuting embeddings of tracial von Neumann algebras}\label{comem}

In this subsection, we show that the above results hold when we further insist that the pair of embeddings have commuting ranges.  While the proof of this fact is rather elementary, it yields some nontrivial consequences.  To set the stage, we first recall the following result from \cite{FGL}:

\begin{thm}[\cite{FGL}]\label{FGL}
Let $N$ be a II$_1$-factor with $N^\mc{U} \cong R^\mc{U}$.  Let $\alpha \in \text{Aut}(N\otimes N)$ be the flip automorphism given by $\alpha(x\otimes y) = y \otimes x$. The following are equivalent.
\begin{enumerate}
\item $N \cong R$;

\item Any two embeddings $\pi,\rho: N \rightarrow R^\mc{U}$, are unitarily equivalent.

\item Given any embedding $\pi: N \otimes N \rightarrow R^\mc{U}, \pi$ and $\pi \circ \alpha$ are unitarily equivalent.

\item Given any embedding $\pi: N \otimes N \rightarrow R^\mc{U}, \pi|_{N\otimes \mb{C}}$ and $\pi \circ \alpha|_{N\otimes \mb{C}}$ are unitarily equivalent.  That is, any pair of commuting embeddings of $N$ into $R^\mc{U}$ are unitarily conjugate.
\end{enumerate}
\end{thm}

\noin Note that Jung's theorem strengthens the equivalence of (1) and (2) in Theorem \ref{FGL} by weakening the hypotheses on $N$ so that $N$ need only satisfy CEP. The proof presented in \cite{FGL} uses the fact that $N^\mc{U} \cong R^\mc{U}$ implies that $N$ is McDuff together with the following result of Connes from \cite{connes}:

\begin{prop}[\cite{connes}]\label{connests}
Let $N$ be a separably acting II$_1$-factor.  The following are equivalent.
\begin{enumerate}
\item $N \cong R$;
\item $N \cong N\otimes R$ and given $x_1,\dots, x_n \in N$, $\varepsilon > 0$ there are $z_1,\dots, z_n \in R$ and a unitary $u \in N\otimes R$ with \[||(x_j \otimes I_R) - u(I_N \otimes z_j)u^*||_2 < \varepsilon\] for every $1 \leq j \leq n$.
\end{enumerate}
\end{prop}

We take this opportunity to show how Kishimoto's Proposition \ref{kishimoto} provides the following ucp analog of the above proposition:

\begin{prop}
Let $N$ be a separably acting II$_1$-factor.  The following are equivalent.
\begin{enumerate}
\item $N \cong R$;
\item $N \cong N\otimes R$ and given $x_1,\dots, x_n \in (N)_{\leq 1}$,$\varepsilon > 0$ there are $z_1,\dots, z_n \in (R)_{\leq 1}$ and a subtracial ucp map $\varphi: N\otimes R \rightarrow N\otimes R$ with \[||(x_j \otimes I_R) - \varphi(I_N \otimes z_j)||_2 < \varepsilon\] for every $1 \leq j \leq n$.
\end{enumerate}
\end{prop}

\begin{proof}
We only need to prove (2) $\Rightarrow$ (1).  We wish to use Kishimoto's proposition.  Fix $\varepsilon> 0$ and $x_1,\dots, x_n \in (N)_{\leq 1}$.

\noin \textbf{Claim.} It suffices to assume that $x_1,\dots, x_n \in (N\otimes \mb{C})_{\leq 1}$. 

\noin \textbf{Proof of claim.} Since $N \cong N \otimes R$ there is an increasing sequence $(N_k)_{k \in \mb{N}}$ of subfactors of $N$, all isomorphic to $N$, with relative commutants isomorphic to $R$.  So we have the decompositions $N = N_k \otimes N_k'$ for every $k \in \mb{N}$.  Given $x_1,\dots, x_n \in (N)_{\leq 1}$ and $\varepsilon > 0$, by the Kaplansky density theorem, there are $k \in \mb{N}$ and  $x_1',\dots, x_n' \in (N_k)_{\leq 1}$ such that $\ds ||x_j - x_j'||_2 < \frac{\varepsilon}{2}$ for $1 \leq j \leq n$.  Thus, we may assume that $x_1,\dots, x_n \in (N \otimes \mb{C})_{\leq 1}.$  

Let $z_1,\dots, z_n \in (R)_{\leq 1}$ and $\varphi: N\otimes R \rightarrow N\otimes R$ be a subtracial ucp map such that \[||x_j - \varphi(I_N\otimes z_j)||_2 \\< \frac{\varepsilon}{2}\] for every $1 \leq j \leq n$.  Since $R$ is hyperfinite, there are a $J\in \mb{N}$ and $z_1',\dots, z_n' \in (\mb{M}_J)_{\leq 1}$ such that $\ds ||z_j - z_j'||_2 < \frac{\varepsilon}{2}$ for every $1 \leq j \leq n$.  Then we have a ucp map $\varphi|_{\mb{C}\otimes \mb{M}_J}: \mb{C}\otimes \mb{M}_J \rightarrow N\otimes R$ with 
\begin{align*}
||x_j - \varphi|_{\mb{C} \otimes \mb{M}_J}(I_N \otimes z_j')||_2 &\leq ||x_j - \varphi(1\otimes z_j)||_2 + ||\varphi(1\otimes z_j) - \varphi(1\otimes z_j')||_2\\
&\leq ||x_j - \varphi(1\otimes z_j)||_2 + ||1\otimes z_j - 1\otimes z_j'||_2\\
&< \varepsilon
\end{align*}  for every $1\leq j \leq n$. Then by Kishimoto's proposition, $N$ is injective.
\end{proof}

Using the results from \S \ref{ucp}, we obtain an upgrade of Theorem \ref{FGL}.  We first prepare with the following elementary lemma:

\begin{lem}\label{alwayscommute}
Let $(N,\tau)$ be a separably acting tracial von Neumann algebra satisfying CEP. Given any pair of embeddings $\pi_1,\pi_2: N\rightarrow R^\mc{U}$, there exist embeddings $\tilde{\pi}_1, \tilde{\pi}_2: N \rightarrow R^\mc{U}$ such that
\begin{enumerate}
    \item $\tilde{\pi}_i$ is unitarily equivalent to $\pi_i$ for $i =1,2$;
    \item $\tilde{\pi}_1(N)$ and $\tilde{\pi}_2(N)$ commute.
\end{enumerate}
\end{lem}

\begin{proof}
Let $\sigma: R\otimes R \rightarrow R$ be a $*$-isomorphism.  Let $\tilde{\pi}_1$ be given by $\tilde{\pi}_1(x) = \sigma^\mc{U}(\pi_1(x) \otimes 1)$, and let $\tilde{\pi}_2$ be given by $\tilde{\pi}_2(x) = \sigma^\mc{U}(1\otimes \pi_2(x))$ where $\sigma^\mc{U}: (R\otimes R)^\mc{U} \rightarrow R^\mc{U}$ is the isomorphism induced by $\sigma$. Following Remark 3.2.4 in \cite{topdyn}, we have that $\tilde{\pi}_i$ is unitarily equivalent to $\pi_i$ for $i = 1,2$. Also, $\tilde{\pi}_1$ and $\tilde{\pi}_2$ clearly have commuting ranges.
\end{proof}

\begin{thm}\label{mr}
Let $(N,\tau)$ be a separably acting tracial von Neumann algebra satisfying CEP, and let $\left\{M_k\right\}$ be a sequence of II$_1$-factors. Let $\alpha \in \text{Aut}(N\otimes N)$ be the flip automorphism given by $\alpha(x\otimes y) = y \otimes x$.The following are equivalent. 

\begin{enumerate}
\item $N$ is amenable;

\item Any two embeddings $\ds \pi,\rho: N \rightarrow \prod_{k\rightarrow \mc{U}}M_k$ are ucp-conjugate

\item Given any embedding $\ds \pi: N \otimes N \rightarrow \prod_{k\rightarrow \mc{U}}M_k, \pi$ and $\pi \circ \alpha$ are ucp-conjugate.

\item Given any embedding $\ds \pi: N \otimes N \rightarrow \prod_{k\rightarrow \mc{U}}M_k, \pi|_{N\otimes \mb{C}}$ and $\pi \circ \alpha|_{N\otimes \mb{C}}$ are ucp-conjugate.  That is, any pair of commuting embeddings of $N$ into $\ds \prod_{k\rightarrow \mc{U}}M_k$ are ucp-conjugate.
\end{enumerate}

\end{thm}

\begin{proof}
It suffices to show that (4) implies (1).  We will use the equivalent formulation of condition (4): any pair of embeddings $\ds \pi_1,\pi_2: N \rightarrow \prod_{k\rightarrow \mc{U}} M_k$ with commuting ranges are ucp-conjugate. Suppose that $N$ is not amenable.  Then by Theorem \ref{jungv2} there are two embeddings $\pi_1,\pi_2: N\rightarrow R^\mc{U}$ that are not ucp-conjugate.  By Lemma \ref{alwayscommute}, we can find $\tilde{\pi}_1,\tilde{\pi}_2: N \rightarrow R^\mc{U}$ with commuting ranges such that $\tilde{\pi}_i$ is unitarily equivalent to $\tilde{\pi}_i$ for $i = 1,2$.  As done before, obtain an embedding $\ds R^\mc{U} \hookrightarrow \prod_{k\rightarrow \mc{U}} M_k$ via embeddings $R \hookrightarrow M_k$.  Then the embeddings $\ds \tilde{\pi}_1,\tilde{\pi}_2: N\rightarrow R^\mc{U} \hookrightarrow \prod_{k\rightarrow \mc{U}} M_k$ have commuting ranges and are not ucp-conjugate.
\end{proof}

\begin{dfn}
Let $N$ be a von Neumann algebra, and let $M$ be a II$_1$-factor.  Let $\HOM(N,M)$ denote the space of all unital $*$-homomorphisms $\pi: N \rightarrow M$ modulo unitary equivalence. Let $[\pi]$ denote the unitary equivalence class of $\pi: N\rightarrow M$.
\end{dfn}

In \cite{topdyn} Brown studied the action of $\text{Out}(N)$ on $\HOM(N,R^\mc{U})$ given by $[\pi] \mapsto [\pi \circ \alpha^{-1}]$.  In particular, it is of interest when $\text{Out}(N)$ acts nontrivially on $\HOM(N,R^\mc{U})$.  Theorem \ref{mr} yields the following result in this context.

\begin{cor}
Let $N$ be a nonamenable separable tracial von Neumann algebra satisfying CEP.  Let $\alpha \in \text{Out}(N\otimes N)$ denote the flip automorphism.  Then $[\pi] \mapsto [\pi \circ \alpha^{-1}]$ is a nontrivial involutive action on $\HOM(N\otimes N, R^\mc{U})$.
\end{cor}

In \S \ref{ucp}, we mentioned Theorem \ref{saasep} as a separable version of Jung's theorem. We can strengthen Theorem \ref{saasep} to the following separable version of Theorem \ref{mr}.  We note that this is a nontrivial consequence that does not readily follow from Theorem \ref{saasep} due to the unavailability of a separable version of Lemma \ref{alwayscommute}.

\begin{thm}\label{wauets}
Let $(N,\tau)$ be a separably acting tracial von Neumann algebra satisfying CEP. Let $\alpha \in \text{Aut}(N\otimes N)$ denote the flip automorphism given by $\alpha(x\otimes y) = y \otimes x$. The following are equivalent. 

\begin{enumerate}
\item $N$ is amenable;

\item For every separably acting II$_1$-factor $M$, any two embeddings $\pi,\rho: N \rightarrow M$, are weakly approximately unitarily equivalent;

\item For every separably acting II$_1$-factor $M$ and any embedding $\pi: N \otimes N \rightarrow M, \pi$ and $\pi \circ \alpha$ are weakly approximately unitarily equivalent;

\item For every separably acting II$_1$ factor $M$ and any embedding $\pi: N \otimes N \rightarrow M, \pi|_{N\otimes \mb{C}}$ and $\pi \circ \alpha|_{N\otimes \mb{C}}$ are weakly approximately unitarily equivalent.
\end{enumerate}
\end{thm}

\begin{proof}
Again, we only need to prove $(4) \Rightarrow (1)$.  And again, we use the following condition equivalent to (4): for every separably acting II$_1$-factor $M$ and any two embeddings $\pi, \rho: N \rightarrow M$ with $\pi(N) \subset \rho(N)'\cap M$, $\pi$ and $\rho$ are weakly approximately unitarily equivalent.

Suppose that $N$ is not amenable.  Then by Theorem \ref{mr}, there are embeddings $\pi, \rho: N \rightarrow R^\mc{U}$ with commuting ranges that are not ucp-conjugate.  By Lemma 3.15 of \cite{saa}, there is a separably acting II$_1$-factor $M$ with $M \subset R^\mc{U}$ such that $\pi(N),\rho(N) \subset M$.  We claim that when considered as embeddings into $M$, $\pi$ and $\rho$ are not weakly approximately unitarily equivalent.  If they were, then by Theorem 3.1 of \cite{autoultra}, $\pi$ and $\rho$ would be unitarily equivalent in $R^\mc{U}$ which in turn implies that $\pi$ and $\rho$ are ucp-conjugate--a contradiction.
\end{proof}

Theorem \ref{mr} sheds some light on Question \ref{waucpq}.  Consider the following property in connection to this question:

\begin{dfn}
A separably acting tracial von Neumann algebra $(N,\tau)$ satisfying CEP has the \emph{ultra ucp lifting property} if for any embedding $N \subset R^\mc{U}$ and ucp map $\varphi: N \rightarrow N$, there is a sequence of ucp maps $\varphi_k: R\rightarrow R$ such that $\pi\circ \varphi(x) = (\varphi_k)_\mc{U}(x) \circ \pi$.  That is, $\pi$ and $\pi \circ \varphi$ are ucp-conjugate.
 \end{dfn}
 
 \noin It follows that if every II$_1$ factor has the ultra ucp lifting property, then Question \ref{waucpq} would be resolved in the affirmative.  It turns out that we can use Theorem \ref{mr} to show that the ultra ucp lifting property is in fact a rare property, at least for tensor-square II$_1$ factors.  The following result was obtained in a conversation with Pieter Spaas; we would like to thank him for allowing us to include it here.
 
  \begin{thm}\label{tsuucplp}
 Let $(N,\tau)$ be a tracial von Neumann algebra satisfying the CEP.  The following are equivalent.
 
 \begin{enumerate}
 	\item $N$ is amenable;
	\item $N\otimes N$ has the ultra ucp lifting property.
 \end{enumerate}
 \end{thm}
 
 \begin{proof}
 If $N$ is amenable, then it has the ucp-lifting property (and thus, so does $N\otimes N$).   Indeed, consider $N \subset R$. Since all embeddings of $N$ into $R^\mc{U}$ are unitarily conjugate, we may assume without loss of generality that $N \subset R \subset R^\mc{U}$ is embedded via the constant-sequence embedding.  Let $\varphi: N \rightarrow N$ be a ucp map.  Then letting $\varphi_k = \varphi$ for every $k$ yields the desired sequence of ucp maps.
 
 If $N$ is not amenable, then by Theorem \ref{mr}, there exists an embedding $\pi: N\otimes N \rightarrow R^\mc{U}$ such that $\pi$ and $\pi \circ \alpha$ are not ucp-conjugate where $\alpha$ denotes the flip automorphism of $N \otimes N$.  Evidently, $N\otimes N$ does not have the ultra ucp lifting property.
 \end{proof}
 
 \noin This result naturally leads us to the following question:
 
 \begin{?}
 Let $(N,\tau)$ be a separable tracial von Neumann algebra satisfying CEP.  If $N$ has the ultra ucp lifting property, then does it follow that $N$ is amenable?
 \end{?}

\subsection{Commuting sofic representations}\label{comsof}

A parallel to Jung's theorem exists in sofic group theory: if any two sofic representations of a sofic group $\Gamma$ are conjugate in the universal sofic group, then $\Gamma$ is amenable. This is due to Elek-Szabo in \cite{elekszabo}. In view of the result on commuting embeddings in Theorem \ref{mr}, it is natural to consider its group analog. This case turns out to be more subtle.  In the tracial von Neumann algebra setting, the self-absorbing nature of $R$ allows us to amplify embeddings while remaining unitarily conjugate. No such self-absorbing behavior is available in the group setting.
%%I don't know what you mean in the following line
In fact, it is a very difficult question to ask whether two sofic embeddings are conjugate if they are conjugate in some amplification.
Fortunately, the work of Elek-Szabo is robust enough to accommodate for this additional commuting condition as we observe in this section. 

Recall the following setup from \cite{elekszabo}. Let $\Gamma$ be a finitely generated group with symmetric generating set $S$. Let $G$ be a finite graph such that each directed edge of $G$ is labeled by an element of $S$. 

\begin{dfn}
We say that $G$ is an \emph{$r$-approximation} of $Cay(\Gamma,S)$ if there exists a subset $W\subseteq V(G)$ such that $|W|>(1-1/r)|V(G)|$ and if $p\in W$ then the $r$-neighborhood of $p$ is rooted isomorphic to the $r$-neighborhood of a vertex of the Cayley graph of $\Gamma$, as edge labeled graphs.
\end{dfn}

\begin{dfn}
We say that $\Gamma$ is \emph{sofic} if for any $r\geq 1$, there exists $r$-approximations of $Cay(\Gamma,S)$ by finite graphs. Call a sequence of graphs $\{G_k\}$ a \emph{sofic approximation} if for all $r\geq 1$, there exists $N$ such that for all $k>N$, $G_k$ is an $r$-approximation of $Cay(\Gamma,S)$.
\end{dfn}

Let $\mc{U}$ be a free ultrafilter on $\mb{N}$. For any sequence $\{n_k\}_{k=1}^{\infty}$, construct a universal sofic group $\mc{S}= \prod_{k\to \mc{U}} \mb{S}_{n_k}$, where $\mb{S}_{n_k}$ is the symmetric group on $n_k$ letters, and the metric used is the normalized Hamming distance $d_n$ on $\mb{S}_n$: \[d_n(\sigma_1,\sigma_2) = n^{-1} \cdot \left|\left\{j \in \left\{1,\dots,n\right\} : \sigma_1(j) \neq \sigma_2(j)\right\}\right|.\] An injective homomorphism $\pi: \Gamma \to \mc{S}$ is called a sofic representation if \[\lim_{k\to \mc{U}}d_{n_k}(g_k, 1_{\mb{S}_{n_k}})= 1\] for every $g \in \Gamma$ where $\pi(g) = (g_k)_\mc{U}$. The following result is standard.

\begin{lem}
From any sofic representation $\pi: \Gamma\to \mc{S}$, we get a sofic approximation $\{G_{n_k}\}$ of $Cay(\Gamma,S)$, and vice versa.
\end{lem}

For any function $f:\mb{N}\to \mb{N}$ and sofic representation $\pi: \Gamma\to \mc{S}$ given by $\pi(g)= (g_k)_{\mc{U}}$ for $g_k\in \mb{S}_{n_k}$, construct the amplification $\ds \pi^f: \Gamma\to \prod_{k\to \mc{U}} \mb{S}_{n_kf(k)}$ given by $\pi^f(g)= (g_k\otimes 1_{f(k)})_{\mc{U}}$ where the tensor product notation is understood by viewing the permutations as permutation matrices and considering the matricial tensor product. Let $\mc{S}^f$ denote $\ds\prod_{k\to \mc{U}} \mb{S}_{n_kf(k)}$. Observe that this maps the space of sofic representations of $\Gamma$ into $\mc{S}$, into the space of sofic representations of $\Gamma$ into $\mc{S}^f$.

\begin{dfn}
Let $E(A)$ denote the edge set in a colored graph $A$. Say that two colored graphs $A$ and $B$ are \emph{$r$-isomorphic} for some $r>0$ if there are subgraphs $A'\subseteq A$ and $B'\subseteq B$ such that $$|E(A')|\geq \left(1-\frac{1}{r}\right)|E(A)|\ \ and \ \ |E(B')|\geq \left(1-\frac{1}{r}\right)|E(B)|$$ and $A'$ is isomorphic to $B'$ as colored graphs. 
\end{dfn}

\begin{dfn}
A sofic approximation $\{G_k\}_{k\in \mb{N}}$ of $Cay(\Gamma, S)$ is called \emph{hyperfinite} if for all $0<\varepsilon<1$, there exists $K_\varepsilon \in \mb{N}$ such that for any $k \in \mb{N}$, there exists a way to erase $\varepsilon|E(G_k)|$ edges of $G_k$ to obtain a graph $G_k'$ having components of vertex size not greater than $K_\varepsilon$. 
\end{dfn}

\begin{thm}[\cite{elekszabo}]\label{nonhypsofrep}
Suppose $\Gamma$ is sofic and non-amenable, then a sofic approximation for $Cay(\Gamma,S)$ is a non-hyperfinite sofic approximation. 
\end{thm}

The following is a routine observation:

\begin{lem}\label{conjlem} Suppose $\{G_k\}$ and $\{H_k\}$ are sofic approximations that induce sofic representations $\pi_1$ and $\pi_2$ respectively, so that $\pi_1$ is conjugate to $\pi_2$. Then, for any $r\geq 1$, there exists a subsequence $\left\{n_k\right\}$ such that for each $k\in \mb{N}$, $G_{n_k}$ is $r$-isomorphic to $H_{n_k}$.
\end{lem}

We are now ready to state and prove the main result of this section.

\begin{thm}
Let $\Gamma$ be a countable sofic group. Let $\alpha \in \text{Aut}(\Gamma \times \Gamma)$ be the flip automorphism given by $\alpha((g,h)) = (h,g)$. The following are equivalent:
    \begin{enumerate}
        \item $\Gamma$ is amenable;
        \item For any universal sofic group $\mc{S}$, any two sofic representations $\pi_1,\pi_2: \Gamma \to \mc{S}$ are conjugate;
        \item For any universal sofic group $\mc{S}$, and any sofic representation $\pi:\Gamma\times\Gamma \to \mc{S}$, $\pi$ and $\pi\circ \alpha$ are conjugate;
        \item For universal sofic group $\mc{S}$ and any sofic representation $\pi: \Gamma\times \Gamma \to \mc{S}$, $\pi|_{\Gamma \times 1}$ and $\pi\circ\alpha|_{\Gamma \times 1}$ are conjugate.
    \end{enumerate}
\end{thm}

\begin{proof}
It suffices to show (4) implies (1). Suppose that $\Gamma$ is non-amenable.  We will produce non-conjugate sofic representations of $\Gamma$. Let $\{G_k\}$ be a non-hyperfinite sofic approximation for $Cay(\Gamma,S)$, with $|V(G_k)|= k$.  By Theorem \ref{nonhypsofrep} there exists $0<\varepsilon<1$ with the following property:
\begin{itemize}
\item[\eqnum\label{eprop}] for every $k\in \mb{N}$, there exists $f(k)> k$ such that if one erases $\varepsilon |E(G_{f(k)})| = \varepsilon f(k)$ edges from $G_{f(k)}$, there is always a component whose vertex set has cardinality strictly greater than $k$.
\end{itemize}
Consider the sofic approximation $\{H_k\}_{n=1}^{\infty}$ where $|V(H_k)|= f(k)$ and $H_k$ is the disjoint union of $\floor{\frac{f(k)}{k}}$ copies of $G_k$ with the remainder being isolated points. Put $\mc{S}: =\prod_{k\to \mc{U}} \mb{S}_{f(k)}$.  Let $\pi_1,\pi_2: \Gamma \rightarrow \mc{S}$ be the sofic representations associated to $\{G_{f(k)}\}$ and $\{H_k\}$ respectively. Now construct sofic representations $\pi_1\otimes 1, 1\otimes \pi_2:\Gamma \to \mc{S}^f$ as follows: for $g\in \Gamma$ define $(\pi_1\otimes 1)(g)= (g_k\otimes 1_{\mb{S}_{f(k)}})$ where $\pi_1(g)= (g_k)_{\mc{U}}, g_k\in \mb{S}_{f(k)}$, and similarly for $h\in \Gamma$, define $(1\otimes \pi_2)(h)= (1_{\mb{S}_{f(k)}}\otimes h_k)$ where $\pi_2(h) = h_k, h_k \in \mb{S}_{f(k)}$. By construction, $\pi_1\otimes 1$ commutes with $1\otimes \pi_2$ and so there exists a sofic representation $\pi: \Gamma\times \Gamma \to \mc{S}^f$ such that $\pi_1\otimes 1= \pi|_{\Gamma\times 1}$ and $1\otimes \pi_2= \pi\circ \alpha|_{\Gamma\times 1}$. 

We now show that $\pi_1\otimes 1$ and $1\otimes \pi_2$ are not conjugate. Consider the sofic approximations associated to $\pi_1\otimes 1$ and $1\otimes \pi_2$ given by $\{\bigsqcup_{i=1}^{f(k)} G_{f(k)}\}$ and  $\{\bigsqcup_{i=1}^{f(k)} H_k\}$ respectively. By Lemma \ref{conjlem}, it suffices to show that there exists $r \geq 1$ such that for all $k$, $\bigsqcup_{i=1}^{f(k)} G_{f(k)}$ is not $r$-isomorphic to $\bigsqcup_{i=1}^{f(k)} H_k$. Fix $r= 1/\varepsilon$ and $k \in \mb{N}$. Let $G'$ be a subgraph of $\bigsqcup_{i=1}^{f(k)} G_{f(k)}$ with $|E(G')| \geq \left(1-\frac{1}{r}\right) \left|E\left(\bigsqcup_{i=1}^{f(k)} G_{f(k)}\right)\right|.$  We will show that $G'$ cannot be isomorphic to any subgraph $H'$ of $\bigsqcup_{i=1}^{f(k)} H_k$ obtained by erasing any number of edges. Indeed, $G'$ is obtained from $\bigsqcup_{i=1}^{f(k)} G_{f(k)}$ by erasing at most $\varepsilon |E(\bigsqcup_{i=1}^{f(k)} G_{f(k)})|= \varepsilon f(k)^2$ edges so it follows that there exists at least one copy of $G_{f(k)}$ in $\bigsqcup_{i=1}^{f(k)} G_{f(k)}$ with at most $\varepsilon f(k)$ edges removed. Since $\varepsilon$ satisfies \eqref{eprop}, there exists a component of $G'$ with more than $k$ vertices. On the other hand all components of $\bigsqcup_{i=1}^{f(k)} H_k$  have no more than $k$ vertices because $\bigsqcup_{i=1}^{f(k)} H_k$ is a disjoint union of $H_k$'s which are each in turn a disjoint union of $G_k$'s with $|G_k|=k$ (and possibly some isolated vertices). So if $H'$ is a subgraph of $\bigsqcup_{i=1}^{f(k)} H_k$ obtained by erasing edges, there is no component of $H'$ with more than $k$ vertices, and thus $H'$ cannot be isomorphic to $G'$. Hence $\bigsqcup_{i=1}^{f(k)} G_{f(k)}$ is not $r$-isomorphic to  $\bigsqcup_{i=1}^{f(k)} H_k$.  Thus $\pi_1\otimes 1$ and $1\otimes \pi_2$ are not conjugate.
\end{proof}

%%\begin{rmk}
%Note that one can produce an uncountable family of pairwise commuting sofic representations that are pairwise non-conjugate, just by modifying the above procedure.
%\end{rmk}

\section{Concluding remarks for the group setting}\label{groups}

The property of stability is based on the following general philosophy: anything that \tql almost satisfies\tqr\, a property, is \tql close\tqr\, to something that precisely satisfies that property. There has been a lot of interest in recent years about the notion of group stability, see \cite{beluth,beclub,arzpau,deglluth,glebriv}. 
%The finitary definition of stability for groups is as follows.

\begin{dfn}[\cite{beluth} Definition 1.1] Let $\mc{G}$ be a class of groups equipped with bi-invariant metrics,  and let $\Gamma$ be a finitely presented group with generators \linebreak $\{s_1,\hdots, s_m\}$ and relations $\{w_1,\hdots, w_r\}$. We say $\Gamma$ is \emph{$\mc{G}$-stable} if given any group $G \in \mc{G}$ with bi-invariant metric $d_G$, for every $\varepsilon>0$ there exists $\delta>0$ such that if $(g_1,\hdots, g_m)\in G^m$ satisfies $d_G(w_j(g_1,\dots, g_m),id_{G})<\delta$ for every $1 \leq j \leq r$  then there is a homomorphsim $\pi: \Gamma \rightarrow G$ with $d_G(\pi(s_j),g_j)<\varepsilon$ for $1 \leq j \leq m$.
\end{dfn}

%\begin{rmk}
%This definition however is very restrictive in the sense that one is not allowed to change the size set on which $\Gamma$ acts. This was noted in (Minsky-Levit-?) and \cite{beclub} and appropriately a modified version called ``flexible stability'' was considered (it also turns out that this is the more natural definition to consider in our context).
%\end{rmk}

\noin Typical examples of the class $\mc{G}$ are the symmetric groups $S_n$ with Hamming metric (this class denoted by $\mc{P}$) and the unitary groups $\mc{U}(\mb{M}_n)$ with the normalized Hilbert-Schmidt distance (this class denoted by $HS$). Taking a cue from Remark \ref{finitary}, we can expand the above finitary definition of stability to apply to groups with infinitely many generators and relations as follows.

\begin{dfn}\label{gpstab}
 Let $\mc{G}$ be a class of groups equipped with bi-invariant metrics, and let $\Gamma$ be a group with generators $\left\{s_1,s_2,\dots\right\}$ and relations $\left\{w_1,w_2,\dots\right\}$.  
 %For any $m \in \mb{N}$, let $X_m$ denote the set $\left\{x_1,\dots, x_m\right\}$, and let $W_m$ denote the set $\left\{w_1,\dots,w_m\right\}$.  
 We have that $\Gamma$ is \emph{$\mc{G}$-stable} if given any group $G \in \mc{G}$ with bi-invariant metric $d_G$, for any $\varepsilon > 0$ and $n \in \mb{N}$, there are $m>n$ and $k \in \mb{N}$ such that if $(g_1,\dots,g_m) \in G^m$ is an $m$-tuple such that $d_G(w_j(g_1,\dots,g_m),\text{id}_{G}) < k^{-1}$ for every $1\leq j \leq m$ then there is a homomorphism $\pi: \Gamma \rightarrow G$ such that $d_G(\pi(s_j),g_j) < \varepsilon$ for every $1 \leq j \leq m$.  
\end{dfn}

Looking at our Definition \ref{ssdef}, we see that one can give a definition of group stability more in keeping with the ultraproduct language of this article.  

\begin{dfn}
Let $\mc{G}$ be a class of groups equipped with bi-invariant metrics, and let $\mc{U}$ be a free ultrafilter on $\mb{N}$.  A countable discrete group $\Gamma$ is \emph{$\mc{G}$-stable} if for any sequence of groups $G_k \in \mc{G}$ with respective bi-invariant metrics $d_k$ and any homomorphism $\ds\pi: \Gamma \rightarrow  \prod_{k\rightarrow \mc{U}}(G_k,d_k)$ (metric ultraproduct), there are homomorphisms $\pi_k: \Gamma \rightarrow G_k$ such that $\pi(g) = (\pi_k(g))_\mc{U}$ for every $g \in \Gamma$.
\end{dfn}

We define self-stability for groups as follows.

\begin{dfn}
Let $\Gamma$ be a countable discrete group. For a bi-invariant metric $d$, consider the metric ultrapower group  $(G,d)^\mc{U}$. The group $(G,d)$ is said to be \emph{self-stable} if any homomorphsim $\pi:G\to (G,d)^\mc{U}$ lifts into homomorphisms $\pi_k:G\to G$ such that $\pi(g)=(\pi_k(g))_{\mc{U}}$ for all $g\in G$. 
\end{dfn}

Many approximation properties for groups (e.g., amenability, Haagerup property, property (T)) have von Neumann algebraic counterparts such that a group satisfying that property is equivalent to its corresponding von Neumann algebra satisfying that property.  That is, $\Gamma$ is amenable if and only if $L(\Gamma)$ is amenable. The following example shows that this is not the case for self-stability.  Note that there are \textit{a priori} many choices for a bi-invariant metric on a given group $\Gamma$; but the appropriate choice for this comparison to make sense is the metric induced on the group $\Gamma$ when $\Gamma$ is considered as a subset of $\mc{U}(L(\Gamma))$ under the trace-norm.

\begin{exmpl}
Fix $n \in \mb{N}$. Let $\Gamma = \mb{F}_n$, and let $d$ be the metric induced on $\mb{F}_n$ when considered a subset of $\mc{U}(L(\mb{F}_n))$ under the trace-norm.  Then it is a direct observation that $(\mb{F}_n,d)$ is self-stable.
\end{exmpl}

\noin Hence $\mb{F}_n,$ is self-stable, but $L(\mb{F}_n)$ is not self-tracially stable by Theorem \ref{tschar}.  The universality of $\mb{F}_n$ is the property that yields self-stability. Note that when passing to its von Neumann algebra, the universality of $\mb{F}_n$ is lost due to the fact that $L(\mb{F}_n)$ is a II$_1$-factor and therefore is simple.  Thus it is of interest to find additional hypotheses on a group $\Gamma$ which, in conjunction with self-stability, imply amenability.

%Thus the question of when self-stability is a property shared by both a group $\Gamma$ and its von Neumann algebra becomes more interesting when we restrict $\Gamma$ to be simple, icc, and hyperlinear.

We also take this opportunity to consider a group-theoretic analog of Corollary \ref{jungv2cor}. As we have seen in the previous section, Elek-Szabo proved in \cite{elekszabo} that up to conjugacy there is exactly one sofic approximation of a finitely generated group $\Gamma$ if and only if $\Gamma$ is amenable. It is natural to ask if one can achieve a generalization similar to the ucp conjugation result of this paper.  We can ask for a hyperlinear analog: 

%We also remark that one could possibly prove a result similar to Corollary \ref{jungv2cor}, like an Elek-Szabo type result for groups with hyperlinear approximations. 

\begin{?}
Let $\Gamma$ be a countable discrete hyperlinear group such that for any two hyperlinear approximations $\pi,\rho:\Gamma \to \mc{U}(R^\mc{U})$ there exists a sequence of positive definite functions $\varphi_k: \mc{U}(R)\to \mc{U}(R) \subset R$ such that $\pi(g)= (\varphi_k)_{\mc{U}}\circ\rho(g)$ for every $g \in \Gamma$; then does it follow that $\Gamma$ is amenable?
\end{?}

\noin And we can ask for a sofic analog:

\begin{?}
Let $\Gamma$ be a countable discrete sofic group such that for any two sofic representations $\pi,\rho:\Gamma \to \prod_{k\to\mc{U}} \mb{S}_{n_k}$ there exists a sequence of positive definite functions $\varphi_k: \mb{S}_{n_k} \to \mb{S}_{n_k} \subset \mb{M}_{n_k}$ such that $\pi(g)= (\varphi_k)_{\mc{U}}\circ\rho(g)$ for every $g \in \Gamma$; then does it follow that $\Gamma$ is amenable?
\end{?}

%In the case of Elek-Szabo, the perturbation one performs on a sofic approximation still remains in $S_n$ at each level of the ultraproduct. However the perturbation performed above twists the sofic approximation outside the symmetry groups at each level but asymptotically lands back in the universal sofic group.

%\section{Questions and concluding remarks}
%\begin{cor}
%If self stability is an axiomatizable property, then the Connes embedding problem has a negative resolution.
%\end{cor}
%\begin{proof}
%\end{proof}

\bibliographystyle{plain}
\bibliography{ssbib}{}

\def\lfhook#1{\setbox0=\hbox{#1}{\ooalign{\hidewidth
  \lower1.5ex\hbox{'}\hidewidth\crcr\unhbox0}}}
\begin{thebibliography}{10}

\bibitem{arzpau}
G.~Arzhantseva and L.~P\u{a}unescu.
\newblock Almost commuting permutations are near commuting permutations.
\newblock {\em J. Funct. Anal.}, 269(3):745--757, 2015.

\bibitem{saa}
S.~Atkinson.
\newblock Convex sets associated to {$C^*$}-algebras.
\newblock {\em J. Funct. Anal.}, 271(6):1604--1651, 2016.

\bibitem{sellin}
S.~Atkinson.
\newblock Some results on tracial stability and graph products.
\newblock \emph{Indiana U. Math. J.}, to appear, arXiv:1808.04664.

\bibitem{genjung}
S.~Atkinson, I.~Goldbring, and S.~Kunnawalkam~Elayavalli.
\newblock Factorial relative commutants and the generalized {J}ung property for
  {II$_1$} factors.
\newblock submitted, arXiv:2004.02293.

\bibitem{beclub}
O.~Becker and A.~Lubotzky.
\newblock Group stability and property ({T}).
\newblock preprint, arXiv:1809.00632v1.

\bibitem{beluth}
O.~Becker, A.~Lubotzky, and A.~Thom.
\newblock Stability and invariant random subgroups.
\newblock preprint, arXiv:1801.08381.

\bibitem{invar}
N.~P. Brown.
\newblock Invariant means and finite representation theory of {$C^*$}-algebras.
\newblock {\em Mem. Amer. Math. Soc.}, 184(865):viii+105, 2006.

\bibitem{topdyn}
N.~P. Brown.
\newblock Topological dynamical systems associated to {${\rm II}_1$}-factors.
\newblock {\em Adv. Math.}, 227(4):1665--1699, 2011.
\newblock With an appendix by Narutaka Ozawa.

\bibitem{brownozawa}
N.~P. Brown and N.~Ozawa.
\newblock {\em {$C^*$}-algebras and finite-dimensional approximations},
  volume~88 of {\em Graduate Studies in Mathematics}.
\newblock American Mathematical Society, Providence, RI, 2008.

\bibitem{choeff}
M.D. Choi and E.~G. Effros.
\newblock Separable nuclear {$C\sp*$}-algebras and injectivity.
\newblock {\em Duke Math. J.}, 43(2):309--322, 1976.

\bibitem{connes}
A.~Connes.
\newblock Classification of injective factors. {C}ases {$II_{1},$} {$II_{\infty
  },$} {$III_{\lambda },$} {$\lambda \not=1$}.
\newblock {\em Ann. of Math. (2)}, 104(1):73--115, 1976.

\bibitem{deglluth}
M.~De~{C}hiffre, L.~Glebsky, A.~Lubotzky, and A.~Thom.
\newblock stability, cohomology vanishing, and non-approximable groups.
\newblock preprint, arXiv:1711.10238.

\bibitem{ding-hadwin}
H.~Ding and D.~Hadwin.
\newblock Approximate equivalence in von {N}eumann algebras.
\newblock {\em Sci. China Ser. A}, 48(2):239--247, 2005.

\bibitem{elekszabo}
G.~Elek and E.~Szabo.
\newblock Sofic representations of amenable groups.
\newblock {\em Proc. Amer. Math. Soc.}, 139(12):4285--4291, 2011.

\bibitem{FGL}
J.~Fang, L.~Ge, and W.~Li.
\newblock Central sequence algebras of von {N}eumann algebras.
\newblock {\em Taiwanese J. Math.}, 10(1):187--200, 2006.

\bibitem{FHS}
I.~Farah, B.~Hart, and D.~Sherman.
\newblock Model theory of operator algebras {III}: elementary equivalence and
  {$\rm II_1$} factors.
\newblock {\em Bull. Lond. Math. Soc.}, 46(3):609--628, 2014.

\bibitem{fms}
T.~Frayne, A.~C. Morel, and D.~S. Scott.
\newblock Reduced direct products.
\newblock {\em Fund. Math.}, 51:195--228, 1962/1963.

\bibitem{glebriv}
L.~Glebsky and L.~M. Rivera.
\newblock Almost solutions of equations in permutations.
\newblock {\em Taiwanese J. Math}, 13(2A):493--500, 2009.

\bibitem{newproof}
U.~Haagerup.
\newblock A new proof of the equivalence of injectivity and hyperfiniteness for
  factors on a separable {H}ilbert space.
\newblock {\em J. Funct. Anal.}, 62(2):160--201, 1985.

\bibitem{hadwin}
D.~Hadwin.
\newblock Free entropy and approximate equivalence in von {N}eumann algebras.
\newblock In {\em Operator algebras and operator theory ({S}hanghai, 1997)},
  volume 228 of {\em Contemp. Math.}, pages 111--131. Amer. Math. Soc.,
  Providence, RI, 1998.

\bibitem{hadshu}
D.~Hadwin and T.~Shulman.
\newblock Tracial {S}tability for {C}*-{A}lgebras.
\newblock {\em Integral Equations Operator Theory}, 90(1):90:1, 2018.

\bibitem{hadshu2}
D.~Hadwin and T.~Shulman.
\newblock Stability of group relations under small {H}ilbert-{S}chmidt
  perturbations.
\newblock preprint, arXiv:1706.08405.

\bibitem{connessol}
Z.~Ji, A.~Natarajan, T.~Vidick, J.~Wright, and H.~Yuen.
\newblock Mip*=re.
\newblock preprint, arXiv:2001.04383.

\bibitem{jung}
K.~Jung.
\newblock Amenability, tubularity, and embeddings into {$R^\omega$}.
\newblock {\em Math. Ann.}, 338(1):241--248, 2007.

\bibitem{mcduff}
D.~McDuff.
\newblock Central sequences and the hyperfinite factor.
\newblock {\em Proc. London Math. Soc. (3)}, 21:443--461, 1970.

\bibitem{mvn4}
F.~J. Murray and J.~von Neumann.
\newblock On rings of operators. {IV}.
\newblock {\em Ann. of Math. (2)}, 44, 1943.

\bibitem{popa}
S.~Popa.
\newblock Independence properties in subalgebras of ultraproduct {$\rm II_1$}
  factors.
\newblock {\em J. Funct. Anal.}, 266(9):5818--5846, 2014.

\bibitem{orbits}
D.~Sherman.
\newblock Unitary orbits of normal operators in von {N}eumann algebras.
\newblock {\em J. Reine Angew. Math.}, 605:95--132, 2007.

\bibitem{autoultra}
D.~Sherman.
\newblock Notes on automorphisms of ultrapowers of {${\rm II}_1$} factors.
\newblock {\em Studia Math.}, 195(3):201--217, 2009.

\bibitem{thom}
A.~Thom.
\newblock Finitary approximations of groups and their applications.
\newblock preprint, arXiv:1712.01052.

\bibitem{voiculescu}
D.~Voiculescu.
\newblock A non-commutative {W}eyl-von {N}eumann theorem.
\newblock {\em Rev. Roumaine Math. Pures Appl.}, 21(1):97--113, 1976.

\bibitem{entropy1}
D.~Voiculescu.
\newblock The analogues of entropy and of {F}isher's information measure in
  free probability theory. {I}.
\newblock {\em Comm. Math. Phys.}, 155(1):71--92, 1993.

\bibitem{entropy}
D.~Voiculescu.
\newblock The analogues of entropy and of {F}isher's information measure in
  free probability theory. {II}.
\newblock {\em Invent. Math.}, 118(3):411--440, 1994.

\bibitem{entropy3}
D.~Voiculescu.
\newblock The analogues of entropy and of {F}isher's information measure in
  free probability theory. {III}. {T}he absence of {C}artan subalgebras.
\newblock {\em Geom. Funct. Anal.}, 6(1):172--199, 1996.

\bibitem{entropy4}
D.~Voiculescu.
\newblock The analogues of entropy and of {F}isher's information measure in
  free probability theory. {IV}. {M}aximum entropy and freeness.
\newblock In {\em Free probability theory ({W}aterloo, {ON}, 1995)}, volume~12
  of {\em Fields Inst. Commun.}, pages 293--302. Amer. Math. Soc., Providence,
  RI, 1997.

\bibitem{entropy5}
D.~Voiculescu.
\newblock The analogues of entropy and of {F}isher's information measure in
  free probability theory. {V}. {N}oncommutative {H}ilbert transforms.
\newblock {\em Invent. Math.}, 132(1):189--227, 1998.

\bibitem{entropy6}
D.~Voiculescu.
\newblock The analogues of entropy and of {F}isher's information measure in
  free probability theory. {VI}. {L}iberation and mutual free information.
\newblock {\em Adv. Math.}, 146(2):101--166, 1999.

\bibitem{wass}
S.~Wassermann.
\newblock Injective {$W\sp*$}-algebras.
\newblock {\em Math. Proc. Cambridge Philos. Soc.}, 82(1):39--47, 1977.

\end{thebibliography}

\end{document}